\documentclass[12pt,a4paper]{article}
\usepackage{latexsym,amssymb,amsmath,a4wide,theorem}
\usepackage[a4paper, total={6in, 10in}]{geometry}
\usepackage[utf8]{inputenc}
\usepackage[T1]{fontenc}
\usepackage{times}
\usepackage{cite}
\usepackage{amsfonts}
\usepackage{graphicx}
\usepackage[all]{xy}
\usepackage{esint}
\usepackage[titletoc,title]{appendix}
\usepackage{indentfirst}
\usepackage{float}
\usepackage{cancel}
\usepackage{booktabs}
\usepackage{makecell}
\usepackage{color}
\usepackage{epstopdf}
\usepackage{hyperref}
\def\sqr#1#2{{\vbox{\hrule height.#2pt
     \hbox{\vrule width.#2pt height#1pt \kern#1pt
           \vrule width.#2pt}
     \hrule height.#2pt}}}

\newenvironment{proof}{\medskip\par\noindent{\bf Proof\/}.\quad}{\qquad
\raisebox{-0.5mm}{\rule{1.5mm}{1mm}}\vspace{6pt}}

\usepackage{mathrsfs}

\newcommand{\R}{\mathbb{R}}

\newtheorem{Thm}{Theorem}[section]

\newtheorem{Lem}{Lemma}[section]
\newtheorem{Pro}{Proposition}[section]

\newtheorem{Rek}{Remark}[section]
\newtheorem{Def}{Definition}[section]

\numberwithin{equation}{section}

\title{{Existence and Multiplicity of $L^2$-Normalized Solutions for a periodic Schr\"{o}dinger system of Hamiltonian type}\thanks{The second author was supported by NSFC of China (12301136). The third author was supported by NSFC of China (12261107) and Yunnan Key Laboratory of Modern Analytical Mathematics and Applications (No. 202302AN360007).}}

\author{Ruowen Qiu$^{1}$,~Yuanyang Yu$^{2}$ and Fukun Zhao$^{1,3}$\thanks{Corresponding author. E-mail address: fukunzhao@163.com}\\
{\small 1.Department of Mathematics, Yunnan Normal University,}\\
{\small 2.Department of Mathematics and Statistics, Yunnan University,}\\
{\small 3.Yunnan Key Laboratory of Modern Analytical Mathematics and Applications}\\
 {\small Cheng Gong, Kunming, 650500}}
\date{}

\date{}

\begin{document}
\maketitle

\noindent{\bf Abstract}\quad In this paper, we study the following
nonlinear Schr\"{o}dinger system of Hamiltonian type
\begin{equation*}
\left\{\begin{array}{l}
-\Delta u+V(x)u=\partial_v H(x,u,v)+\omega v,\  x \in \mathbb{R}^N, \\
-\Delta v+V(x)v=\partial_u H(x,u,v)+\omega u,\  x \in \mathbb{R}^N, \\
\displaystyle\int_{\R^N}|z|^2dx=a^2,
\end{array}\right.
\end{equation*}
where the potential function $V(x)$ is periodic, $z:=(u,v):\R^N\rightarrow \R\times\R$,  $\omega\in \R$ appears as a Lagrange multiplier, $a>0$ is a prescribed constant.
The existence and multiplicity of $L^2$-normalized
solutions for the above Schr\"{o}dinger system are obtained.
The proof combines a constrained variational reduction,
compactness estimates, and a Ljusternik--Schnirelmann
type multiplicity theorem.
In addition, a bifurcation result is also given. \\

\numberwithin{equation}{section}

\tableofcontents

\section{Introduction and main result}
The coupled nonlinear Schr\"{o}dinger system
\begin{equation*}
\left\{\begin{array}{l}
-i \hbar \frac{\partial \phi_1}{\partial t}=\frac{\hbar^2}{2 m} \Delta \phi_1-\phi_1-V(x) \phi_2+f(x, \phi) \phi_2, x \in \mathbb{R}^N, t>0, \\
-i \hbar \frac{\partial \phi_2}{\partial t}=\frac{\hbar^2}{2 m} \Delta \phi_2-\phi_2-V(x) \phi_1+f(x, \phi) \phi_1, x \in \mathbb{R}^N, t>0,
\end{array}\right.
\end{equation*}
arises in some physical problems, especially in the fields of nonlinear optics and Bose-Einstein condensates theory \cite{Ambrosetti-Colorado2007JLMS,Lin-Wei2006JDE,Sirakov-Soares2010TAMS}, where $\phi=\left(\phi_1, \phi_2\right)$, and $\phi_j(t,x)$ for $j=1,2$ represent the wave functions of the states of an electron, $i$ is the imaginary unit, $m$ is the mass of a particle, $\hbar$ is the Planck constant, $V(x)$ is the potential function, and $f$ is the coupled nonlinear function modeling various types of interaction
effects among many particles and satisfies $f(x,e^{i\theta}\phi)=f(x,\phi)$ for $\theta\in[0,2\pi]$. We will
be interested in standing wave solutions of the system, that is, solutions
in the form
$$\phi_1(t, x)=e^{i \omega t} u(x),\ \phi_2(t, x)=e^{i \omega t} v(x).$$
Then, one can see that $z=(u,v)$ satisfies the following elliptic systems 
\begin{equation*}
\left\{\begin{array}{l}
-\varepsilon^2 \Delta u+\lambda u+V(x) v=f(x, z) v ,\ x\in\ \mathbb{R}^N, \\
-\varepsilon^2 \Delta v+\lambda v+V(x) u=f(x, z)u,\ x\in\ \mathbb{R}^N,
\end{array}\right.
\end{equation*}
where $\varepsilon^2=\hbar^2 / 2 m$, $ \lambda=(1+\hbar \omega)$.
\par
In this paper, we focus our attention on the existence and multiplicity of the normalized
solutions for the following Schr\"{o}dinger system of Hamiltonian type
\begin{equation}\label{eq1.1}
\left\{\begin{array}{l}
-\Delta u+V(x)u=\partial_v H(x,u,v)+\omega v,\  x \in \mathbb{R}^N, \\
-\Delta v+V(x)v=\partial_u H(x,u,v)+\omega u,\  x \in \mathbb{R}^N,
\end{array}\right.
\end{equation}
satisfies $\displaystyle
\int_{\R^N}|z|^2dx=a^2$, where $z:=(u,v):\R^N\rightarrow \R\times\R$, $N\geq2$, $\omega\in \R$ appears as a Lagrange multiplier, and $a>0$ is a prescribed constant.
\par
Recently, the problem of finding $L^2(\R^N)$-normalized solutions of Schr\"{o}dinger equation has attracted much attention. This kind of condition is natural, for instance, in quantum mechanics, 
where one looks for particles which have unit charge. When the system \eqref{eq1.1} satisfy certain symmetry conditions, it can be simplified as the following Schr\"{o}dinger equation
\begin{equation}\label{eq1.4}
- \Delta u+V(x) u=f(x,u)+\lambda u, \ u \in H^1\left(\mathbb{R}^N\right).
\end{equation}
We first consider the autonomous equation without a potential:
\begin{equation}\label{eq1.5}
\left\{\begin{array}{l}
-\Delta u+\lambda u=f(u), \\
\displaystyle\int_{\mathbb{R}^N} u^2=a.
\end{array}\right.
\end{equation} 
In the pioneering paper~\cite{Jeanjean1997NA},
Jeanjean obtained a radial solution at a mountain
pass value
for a class of mass supercritical and Sobolev subcritical
nonlinearities.

Afterward, a multiplicity result was established
by Bartsch and de Valeriola in
\cite{Bartsch-deValeriola2013AM}.
Hirata and Tanaka~\cite{Hirata-Tanaka2019ANS}
developed mountain pass and symmetric mountain
pass methods for autonomous scalar equations
under an $L^2$ constraint.
For more general nonlinearities, see
\cite{Jeanjean-Lu2020CVPDE,Bieganowski-Mederski2021JFA}.
For a related equation involving critical local and nonlocal
terms, Feng and Liu~\cite{Feng-Liu2024FM} established infinitely
many normalized solutions for sufficiently small masses.

\par

We next consider \eqref{eq1.4} when $V$ is nonconstant
or $f$ depends explicitly on $x$.
Chen and Tang~\cite{Chen-Tang2020JGA} obtained normalized
ground states for a class of non-autonomous equations.
Ikoma and Yamanobe~\cite{Ikoma-Yamanobe2024ANS}
studied equations with non-autonomous nonlinearities
and obtained multiplicity in the radial setting using
a Palais--Smale--Cerami condition.
Further existence and multiplicity results for equations
with potentials or non-autonomous nonlinearities can be found in
\cite{Zhong-Zou2023DIE,Ikoma-Miyamoto2020CVPDE,
Yang-Qi-Zou2022JGA,Liu-Zhao2024CVPDE}
and the references therein.

\par
Related results extend to quasilinear and fractional equations.
Jeanjean, Luo and Wang~\cite{Jeanjean-Luo-Wang2015JDE}
obtained multiple normalized solutions for a quasilinear
Schr\"odinger equation.
Nguyen and R\u{a}dulescu~\cite{Nguyen-Radulescu2024FM}
established multiplicity for fractional $p$-Laplacian equations
with prescribed $L^p$ norm using
Lusternik--\allowbreak Schnirelmann category theory.
Deng and Wu~\cite{Deng-Wu2024FM} treated fractional
Sobolev critical growth by subcritical approximation.
\par

\par
There are also several works on the existence
of normalized solutions for Schr\"{o}dinger systems.
These studies mainly concern systems of gradient
type, with positive definite quadratic parts
in the associated energy functionals. In \cite{Bartsch-Jeanjean-Soave2016JMPA}, Bartsch, Jeanjean and Soave studied the following  system of coupled elliptic equations
\begin{equation}\label{eq1.6}
\left\{\begin{array}{l}
-\Delta u-\lambda_1 u=\mu_1 u^3+\beta u v^2,\ x\in\R^3, \\
-\Delta v-\lambda_2 v=\mu_2 v^3+\beta u^2 v,\ x\in\R^3, 
\end{array} \right.
\end{equation}
together with the conditions
$$\int_{\mathbb{R}^3} u^2dx=a_1^2 \ \text { and } \  \int_{\mathbb{R}^3} v^2dx=a_2^2,$$
where $a_1$, $a_2$, $\mu_1$, $\mu_2>0$ are fixed quantities. Under suitable conditions on the coupling parameter $\beta>0$, the authors proved the existence of solutions $(\lambda_1,\lambda_2,u,v)$ of \eqref{eq1.6} with positive, radially symmetric components $u$ and $v$. Subsequently, the multiplicity of normalized solutions for \eqref{eq1.6} was studied in \cite{Bartsch-Soave2019CVPDE}. In \cite{Noris-Tavares-Verzini2015DCDS}, Noris, Tavares and Verzini studied the existence of standing waves with prescribed $L^2$-mass for the cubic Schr\"{o}dinger system with trapping potentials in $\R^N$ or in bounded domains, where $N\leq3$.
 Other interesting results for the existence and multiplicity of normalized solutions for related Schr\"{o}dinger systems can be found in \cite{Gou-Jeanjean22018Nonlinearity,Bartsch-Soave2017JFA,Jia-Li-Luo2020DCDS,Lu2020NA} and references therein. Song, Wang and Chen~\cite{Song-Wang-Chen2026BSM}
studied normalized solutions of a planar
Kirchhoff--Schr\"odinger--Poisson system with
Moser--Trudinger critical growth, using variational
methods and mountain pass level estimates.
\par
Unlike the papers mentioned above, the energy functional of \eqref{eq1.1} is strongly indefinite since the system is of Hamiltonian type. 
Buffoni and Jeanjean~\cite{Buffoni-Jeanjean1993}
studied the equation
\begin{equation}\label{eq1.6++}
-\Delta u+V(x)u=a(x)|u|^{p-2}u+\lambda u,
\qquad x\in\R^N,
\end{equation}
where $V$ is periodic, $a\in L^\infty(\R^N)$
is nonnegative, and $2<p<2^*$.
Under additional assumptions on $a$ and $p$, they
established spectral-edge bifurcation through
a constrained minimax construction on $L^2$
spheres, with $\lambda$ arising as a Lagrange
multiplier.
 In \cite{Heinz-Kupper-Stuart1992JDE}, Heinz, K\"upper and Stuart constructed the test functions required by the abstract bifurcation theorem with the help of the Bloch wave of the periodic Schr\"{o}dinger equation, and then successfully proved that lower or upper end-points of the continuous spectrum are bifurcation points. 
Still considering the case where $\lambda$ is located in the spectral gap of the periodic operator $\mathcal{S}=-\Delta+V$, Heinz \cite{Heinz1993NA} applied a generalized multiple critical points theorem to obtain infinitely many pairs $(\lambda,\pm u_j)\in\{\lambda\}\times (H^1(\R^N)\setminus\{0\})$ of weak solutions for a class of equations \eqref{eq1.4} with non-autonomous nonlinearity.
In \cite{Buffoni-Esteban2006ANS}, the authors used the penalization method inspired by \cite{Esteban-Sere1999} to deal with the normalized solutions of different problems with strongly indefinite structure. Recently, via the idea developed in \cite{Buffoni-Jeanjean1993}, \cite{Esteban-Sere1999} and \cite{Buffoni-Esteban2006ANS}, some progress was made in the study of $L^2$-normalized solutions of some strongly indefinite variational problems. See \cite{Ding-Yu-Zhao2023JGA} and \cite{Yu2023arXiv} for Dirac equations, and \cite{Guo-Yu2024AJMA} for the first order Hamiltonian systems.
\par
Motivated by \cite{Buffoni-Jeanjean1993,Heinz1993NA,
Heinz-Kupper-Stuart1992JDE,Ding-Yu-Zhao2023JGA},
we study the existence and multiplicity of normalized
solutions of \eqref{eq1.1} for small prescribed mass,
together with their behavior near the spectral edge.
To state our main results, we impose the following
assumptions on $H(x,u,v)$:
\begin{itemize}
\item[$(H_0)$]  $H(x,z)=G(x,|z|):=\displaystyle\int_0^{|z|} g(x,t)tdt$,~$\forall~ z:=(u,v) \in \mathbb{R} \times \mathbb{R}$;
\item[$(H_1)$] $g(\cdot,t)$ is measurable for every $t\geq0$, $g(x,\cdot) \in C^1(\mathbb{R}^{+}, \mathbb{R}^{+})$, and $g(x,0)=0$ for any $x\in\R^N$; 
\item[$(H_2)$] there exist $p,q\in\R$ with $2<p\leq q<  {2^*}$ such that for almost all $x\in\R^N$, $\displaystyle\frac{g(x,t)}{t^{q-2}}$(resp. $\displaystyle\frac{g(x,t)}{t^{p-2}}$) is nonincreasing (resp. nondecreasing) with respect to $t$ on $(0,\infty)$;
\item[$(H_3)$] $K(x):=g(x,1)\in L^\infty(\R^N)$ and $\lim\limits_{R\rightarrow\infty}\mathop{\text{ess sup}}\limits_{|x|\geq R} K(x)=0$;
\item[$(H_4)$] There exist constants $L,t_0>0$, $\alpha\in(0,2+\frac{4}{N})$ and  $\tau\in(0,\frac{4-N(\alpha-2)}{2})$ such that
$$
 G(x,t)\geq L|x|^{-\tau}t^\alpha \text { a.e. on } S \text { and }0\leq t\leq t_0,
$$
\end{itemize}
where $S=\{tx \mid t \geq 1$ and $ x \in B_d(x_0)\}$ for some $x_0 \in \R^N$ and $0<d<|x_0|$. Here $2^*=2N/(N-2)$ for $N\geq3$ and $2^*=\infty$ for $N=2$.

The potential $V(x)$ is assumed as follows:
\begin{itemize}
\item[$(V_1)$]
$V\in C(\mathbb R^N,\mathbb R)$ is $1$-periodic
in each $x_i$ for $i=1,\cdots, N$. Moreover, for $\mathcal S=-\Delta+V$, there exists $m>0$
such that
$
\sigma(\mathcal S)\cap(-m,m)=\varnothing$,
$\{-m,m\}\subset\sigma(\mathcal S),$ where $\sigma(\mathcal S)$ denotes the spectrum of the operator $\mathcal S$.
\end{itemize}

As a consequence of assumptions $(H_0)$-$(H_4)$, for almost every $x\in\R^N$ and $t>0$, we know that 
\begin{equation}\label{eqg_1}
0\leq(p-2)g(x,t)\leq g^\prime(x,t) t\leq(q-2) g(x,t),
\end{equation}
where $g^\prime(x,t)=\displaystyle\frac{\partial g(x,t)}{\partial t} $.
This then indicates that
\begin{equation}\label{eqg_2}
\frac{1}{q}g(x,t)t^2\leq G(x,t)\leq \frac{1}{p}g(x,t)t^2,
\end{equation}
and 
\begin{equation}\label{eqg_3}
\begin{split}
s^pG(x,t)&\leq G(x,st)\leq s^q G(x,t) \ \text {for}\  s\geq 1,\\
s^qG(x,t)&\leq G(x,st)\leq s^p G(x,t) \ \text {for}\ 0<s\leq 1.
\end{split}
\end{equation}
We now apply $\eqref{eqg_1}$-$\eqref{eqg_3}$ to conclude that there exists $C>0$ such that 
\begin{equation}\label{eqg_4}
G(x,t)\leq C(|t|^p+|t|^q),  \ \text{for all}\ t\in \R^+ \ \text{and}\ a.e.\  x\in \R^N.
\end{equation}

\begin{Rek}
By $(H_2)$ and $(H_4)$,
\[
L|x|^{-\tau}t^\alpha
\leq G(x,t)
\leq \frac{K(x)}{p}t^p,
\qquad 0<t\leq\min\{1,t_0\},
\]
for a.e.\ $x\in S$.
Thus, $p\leq\alpha$. More precisely,
\[
2<p\leq\alpha<2+\frac{4}{N}.
\] We now give an example satisfying $(H_0)$--$(H_4)$, let $N=3$ and
\[
g(x,t)=(1+|x|^2)^{-1/8}(t+t^2).
\]
Then $(H_1)$--$(H_4)$ hold with $
p=\alpha=3$, $q=4$, $\tau=\frac14.$
Indeed, $
K(x)=2(1+|x|^2)^{-1/8}\to0$ as $|x|\to\infty,$
and, since $\inf\limits_{x\in S}|x|>0$,
\[
G(x,t)
=(1+|x|^2)^{-1/8}
\left(\frac{t^3}{3}+\frac{t^4}{4}\right)
\ge c|x|^{-1/4}t^3,
\qquad x\in S,\ t\ge0,
\]
for some $c>0$.
\end{Rek}

Before presenting the main results of this paper, let's review the following definitions of bifurcation point. 
\begin{Def}\label{Def1.1}
The number $\omega_0$ is said to be a bifurcation point of system \eqref{eq1.1} if there exists
$$\{(\omega_n,z_n)\}\subset \bigg\{(\omega,z)\in \R \times H^1(\R^N,\R^2)\mid z\neq0 \ \text{and}\ Tz=g(x,|z|)z+\omega z\bigg\}$$
such that $\omega_n<\omega_0$ for all $n\in \mathbb{N}$, $\omega_n\rightarrow\omega_0$, $\lim\limits_{n\rightarrow\infty}\|z_n\|_{H^1}=0$, where the definition of $T$ is provided in Sect.2.
\end{Def}
\begin{Thm}\label{Thm1.1}
Assume that $(V_1)$ and $(H_0)$--$(H_4)$ hold.
Then, for every sufficiently small $a>0$, there exists
a solution $(\omega_a,z_a)$ of system \eqref{eq1.1}
such that
$$
\int_{\mathbb{R}^N}|z_a|^2\,dx=a^2.
$$
Moreover, with $m>0$ as in $(V_1)$, we have
$\omega_a<m$ and
\[
\omega_a\to m,
\qquad
\|z_a\|_{H^1}\to0
\quad\text{as }a\to0^+.
\]
In particular, $m$ is a bifurcation point of
system \eqref{eq1.1}.
\end{Thm}

We now turn to the multiplicity of $L^2$-normalized
solutions of system~\eqref{eq1.1}.
By applying a Ljusternik--Schnirelmann type multiplicity
theorem, we obtain the following result.
\begin{Thm}\label{Thm1.2}
Assume that $(V_1)$ and $(H_0)$--$(H_4)$ hold. There exists $a_0>0$ such that, for every $a\in(0,a_0)$, system \eqref{eq1.1} possesses a sequence of distinct solutions $\{(\omega_a^i,z_a^i)\}_{i=1}^\infty$ satisfying
    \begin{equation*}
    \int_{\R^N}|z_a^i|^2dx=a^2.
\end{equation*}
Moreover,
    $m$ is a bifurcation point on the left of
system \eqref{eq1.1}. 
\end{Thm}

\begin{Rek}
The bifurcation conclusion in Theorem~1.1 follows from
\[
\|z_a\|_2=a,\qquad
0<m-\omega_a\le Ca^{p-2},\qquad
\|z_a\|_{H^1}\le Ca,
\]
where $C>0$ is independent of sufficiently small $a>0$.
Thus $(\omega_a,z_a)\to(m,0)$ in $\mathbb R\times H^1$
as $a\to0^+$, with $\omega_a<m$.

Our test functions are adapted from the Bloch-wave
construction of Heinz, K\"upper and Stuart
\cite{Heinz-Kupper-Stuart1992JDE}, who studied
spectral-edge bifurcation using abstract bifurcation theory.
Here we prescribe the $L^2$ norm and obtain the frequency
as a Lagrange multiplier.
Theorem~1.2 further gives infinitely many distinct normalized
solutions at each sufficiently small fixed mass,
with possibly different Lagrange multipliers.
\end{Rek}

Our results concern a periodic potential satisfying $(V_1)$
and non-autonomous nonlinearities,
without imposing radial symmetry.
These assumptions differ from those used in the
autonomous or radial multiplicity results
\cite{Hirata-Tanaka2019ANS,Ikoma-Yamanobe2024ANS}
and from the condition
$\inf\limits_{x\in\mathbb{R}^N}V(x)
<\liminf\limits_{|x|\to\infty}V(x)$
on the potential imposed in
\cite{Nguyen-Radulescu2024FM}.
In addition, under $(H_2)$, our nonlinearity has Sobolev
subcritical growth, while its upper growth
exponent $q$ may satisfy $q\geq 2+\frac{4}{N}$.
This differs from the Sobolev critical growth
considered in \cite{Feng-Liu2024FM,Deng-Wu2024FM}
and the Moser--Trudinger critical growth in
\cite{Song-Wang-Chen2026BSM}.
For the local power model \eqref{eq1.6++},
Buffoni, Esteban and S\'er\'e
\cite{Buffoni-Esteban2006ANS} studied normalized
solutions with $\|u\|_2=1$ and the Lagrange
multiplier $\lambda$ lying in an internal
spectral gap of $-\Delta+V$.
Our results yield infinitely many normalized
solutions of system \eqref{eq1.1} at each
sufficiently small fixed mass, with corresponding
Lagrange multipliers satisfying $0<\omega<m$,
where $m$ is as in $(V_1)$.
Such an infinite multiplicity conclusion is not
established in their work.

This paper is organized as follows.
Section 2 introduces the variational framework
based on the spectral decomposition of $T$.
In Section 3, we develop a constrained reduction
preserving the prescribed $L^2$ norm and prove
Theorem~\ref{Thm1.1}.
In Section 4, $(H_3)$ and the multiplier estimates
yield the Palais--Smale condition below $ma^2/2$,
while Bloch-wave test functions and $(H_4)$ ensure
that all genus minimax levels lie below this threshold.
This yields Theorem~\ref{Thm1.2}, with a smallness
threshold for the mass independent of the genus index.

\section{The variational framework}
Throughout the paper, $c$, $c_i$, $C$ or $C_i$ $(i=1,2,\cdots)$ are some positive constants may change from line to line. For $1\leq q\leq\infty$, $L^q:=L^q(\R^N,\R^2)$ and $\|\cdot\|_{q}$ denotes the usual norm of the space $L^q$; $z_n\rightharpoonup z$ and $z_n\rightarrow z$ mean the weak and strong convergence, respectively,
 as $n\rightarrow\infty$. In addition, for $r>0$ and $x\in\R^N$, we denote by $B_r(x)$ the open ball in $\R^N$ of radius $r$ centered at $x$, and let $B_r=B_r(0)$. 
 
First of all, we rewrite \eqref{eq1.1} in order to establish the variational framework.  Set
$$
\mathcal{J}=\left(\begin{array}{ll}
0 & I \\
I & 0
\end{array}\right)\ \text { and }
T:=\mathcal{J}\mathcal{S} =\left(\begin{array}{ll}
\ \ \ \  0 & -\Delta+V \\
-\Delta+V & \ \ \ \  0
\end{array}\right).
$$
Then \eqref{eq1.1} can be read as
\begin{equation}\label{eq2.1}
Tz=g(x,|z|)z+\omega z\ \text{for}\  z=(u,v).
\end{equation}
Under $(V_1)$, $T$ is selfadjoint on $L^2$ with domain  
$$\mathcal{D}(T)=H^2:=H^2(\R^N,\R^2).$$
Denote by $\sigma(T)$ the spectrum of the operator $T$.
Let $m>0$ be as in $(V_1)$, and define
\[
\mathcal{P}
=\frac{1}{\sqrt{2}}
\begin{pmatrix}
I&I\\
I&-I
\end{pmatrix}.
\]
Then $\mathcal{P}$ is unitary and satisfies
\[
\mathcal{P}^{*}T\mathcal{P}
=
\begin{pmatrix}
\mathcal{S}&0\\
0&-\mathcal{S}
\end{pmatrix}.
\]
Hence,
\[
\sigma(T)
=\sigma(\mathcal{S})\cup\bigl(-\sigma(\mathcal{S})\bigr),
\]
so $\sigma(T)$ is symmetric with respect to the origin.
Moreover, $(V_1)$ yields
\begin{equation}\label{eq_pu}
\sigma(T)\cap(-m,m)=\varnothing,
\qquad
\{-m,m\}\subset\sigma(T).
\end{equation}

Let $\{E_\lambda\}_{\lambda \in \mathbb{R}}$ be the spectral family of the $T$. According to \cite[Chapter 4, Theorem 3.3]{Edmunds-Evans1987}, $T$ has a polar decomposition 
$$T=U|T|=|T|U,$$
where $U=I-2E_0$ and $|T|$ denotes the absolute value of $T$. Since $0\notin\sigma(T)$ by \eqref{eq_pu}, $L^2$ possesses the orthogonal decomposition 
$$
L^2=L^{-} \oplus L^{+},\ \text { where } L^{ \pm}=\left\{z \in L^2 \mid U z= \pm z\right\}.
$$
In order to seek for solutions for \eqref{eq1.1}, we denote by $E:=\mathscr{D}(|T|^{\frac{1}{2}})$ with the inner product
$$
\langle z_1, z_2\rangle:=\left( |T|^{\frac{1}{2}} z_1, |T|^{\frac{1}{2}} z_2\right)_{2}\  \text{for any}\ z_1,z_2\in E,
$$
where $(\cdot,\cdot)_2$ denotes the usual inner product
in $L^2$, and $|T|^{1/2}$ denotes the positive
square root of $|T|$. For any $z\in E$, the induced norm $\|z\|^2:=\langle z, z\rangle$. Clearly, $E$ is a Hilbert space. 
Moreover, by the complex interpolation theory which can see \cite{TriebelBook1978}, we know that 
$$
E=\left[H^2\left(\mathbb{R}^N, \mathbb{R}^2\right), L^2\right]_{\frac{1}{2}}=H^1\left(\mathbb{R}^N, \mathbb{R}^2\right).
$$
The orthogonal decomposition to $L^2$ induces the following associated  decomposition:
$$
E=E^{+} \oplus E^{-},\  \text { where } E^{ \pm}=H^1\left(\mathbb{R}^N, \mathbb{R}^2\right)\cap L^{ \pm}.
$$
Using the polar decomposition and self-adjointness of $T$, for any $z^+\in E^+\subset L^+$ and $z^-\in E^-\subset L^-$, there holds 
$$(z^+,z^-)_2=0\ \text{and}\ \langle z^+,z^-\rangle=0,$$
and we conclude that 
\begin{equation}\label{eq2.2}
\langle Tz,z\rangle_{2}=\|z^+\|^2-\|z^-\|^2,\qquad z\in E.
\end{equation}

Recall that if $z=(u, v) \in E$ is a weak solution of system \eqref{eq1.1}, then there holds
$$
\int_{\mathbb{R}^N}(\nabla u\nabla\eta_2+V(x)u \eta_2+\nabla v \nabla \eta_1+V(x)v \eta_1)dx-\omega\int_{\mathbb{R}^N}
z\eta dx -\int_{\mathbb{R}^N} g(x,|z|)z\eta dx =0
$$
for any $\eta=(\eta_1, \eta_2) \in E$. The functional associated with system \eqref{eq1.1} is
\begin{equation}\label{eq2.3}
J_\omega(z)=J_\omega(u, v)=\int_{\mathbb{R}^N}(\nabla u \nabla v+V(x)u v)dx-\frac{\omega}{2}\int_{\mathbb{R}^N} |z|^2dx-\Psi(z),
\end{equation}
where
$$\Psi(z)=\int_{\R^N}G(x,|z|)dx.$$
\eqref{eq2.2} implies that the functional $J_{\omega}$ defined in \eqref{eq2.3} can be rewritten in a standard way
\begin{equation}\label{eq2.2+}
J_\omega(z)=\frac{1}{2}\left(\left\|z^{+}\right\|^2-\left\|z^{-}\right\|^2\right)-\frac{\omega}{2}\int_{\mathbb{R}^N} |z|^2dx-\Psi(z)
\end{equation}
for $z=z^++z^-\in E$.
It follows by standard arguments that $J_\omega\in C^2(E,\R)$. In addition, it can be understood that a solution of \eqref{eq1.1} with $\|z\|_2=a$ can be obtained as a constrained critical point of the functional
$$J_0(z)=\frac{1}{2}\left(\left\|z^{+}\right\|^2-\left\|z^{-}\right\|^2\right)-\Psi(z)$$
on the constraint
$$S_a=\{z\in E\mid\ \displaystyle\int_{\R^N}|z|^{2}dx=a^2\},$$
where $a>0$ is a prescribed constant. 
Furthermore, \eqref{eq_pu} implies that
$$
\sigma(|T|)\subset[m,\infty).
$$
Since $\|z\|^2=\||T|^{1/2}z\|_2^2$, the spectral theorem yields
\begin{equation}\label{eq2.5}
m\|z\|_2^2\le\|z\|^2
\qquad\text{for all }z\in E.
\end{equation}
\section{Existence of normalized solutions}
In this section, we will introduce the following set introduced in \cite{Buffoni-Jeanjean1993} to overcome the difficulty that the norm $\|\cdot\|_2$ and $\|\cdot\|$ are not equivalent on $E^+$, i.e.
\begin{equation}\label{eq3.1}
W:=\{w\in E^+\mid\|w\|<\sqrt{m+1}\|w\|_{2}\}.
\end{equation}
It is clear that $W$ is an open subset of $E^+$.
For $w\in E^+$ with $w\neq0$, define
\begin{equation*}
M(w):=\left\{z\in E\mid  \|z\|_{2}=\|w\|_{2}, \
z^+=\frac{\|z^+\|_{2}}{\|w\|_{2}}w\ \ \mbox{and} \ \ \|z^-\|\leq
\frac{\sqrt{m}\|w\|_{2}}{2} \right\}.
\end{equation*}
For the approximate image of $M(w)$, the interested reader can see \cite[Figure 1]{Yu2023arXiv}. In the following proof, we always assume that $(V_1)$ and $(H_0)$-$(H_4)$ holds.

\subsection{The test-functions for \eqref{eq1.1}}
Consider the following eigenvalue problem for elliptic system:
\begin{equation}\label{EQ}
\begin{cases}
-\Delta u+V(x)u=\lambda v, & x\in\R^N,\\
-\Delta v+V(x)v=\lambda u, & x\in\R^N.
\end{cases}
\end{equation}
Let $\lambda=m$, where $m>0$ is given by $(V_1)$.
Since $\pm m\in\sigma(\mathcal{S})$, Floquet--Bloch
theory (see \cite[Theorem~5.5(6)]{Kuchment2016}) yields
nontrivial solutions
\[
b_\pm(x)=e^{i\theta_\pm\cdot x}p_\pm(x),
\qquad
\theta_\pm\in\R^N,
\qquad
\mathcal{S}b_\pm=\pm m b_\pm,
\]
where $p_\pm$ are $\mathbb{Z}^N$-periodic.
By elliptic regularity, $p_\pm\in C^1(\R^N,\mathbb{C})$.
Since $b_\pm\not\equiv 0$, we may choose
\[
\psi_\pm\in
\{\operatorname{Re}b_\pm,\operatorname{Im}b_\pm\}
\setminus\{0\}.
\]
Since $V$ is real-valued, we have
$\psi_\pm\in H_{\mathrm{loc}}^2(\R^N)\cap C^1(\R^N)$ and
\begin{equation}\label{SEQ}
\mathcal{S}\psi_+=m\psi_+,
\qquad
\mathcal{S}\psi_-=-m\psi_-
\quad\text{in }\R^N.
\end{equation}
The periodicity and $C^1$ regularity of $p_\pm$
imply that $\psi_\pm$ and their first-order derivatives
are bounded and uniformly almost periodic
in the sense of \cite{Besicovitch-1955}.

Define
\[
\nu:=(\phi,\psi)
:=\frac{1}{\sqrt{2}}
\bigl(\psi_++\psi_-,\,\psi_+-\psi_-\bigr).
\]
By \eqref{SEQ},
\[
\mathcal{S}\phi=m\psi,
\qquad
\mathcal{S}\psi=m\phi,
\]
and hence $T\nu=m\nu$.
Moreover,

\[
|\nu|^2=\phi^2+\psi^2=\psi_+^2+\psi_-^2.
\]
Thus $\nu$ is a nontrivial real-valued uniformly almost
periodic solution of \eqref{EQ} with $\lambda=m$.
Furthermore, $\nu$ and its first-order derivatives
are bounded.
Since $\nu\notin L^2(\R^N,\R^2)$, we use the truncation
argument in \cite{Heinz-Kupper-Stuart1992JDE}
to construct the desired sequence.
For $k>0$ and $x\in\R^N$, we set
\begin{equation}\label{eq3.15}
z_k(x)
=k^{-\frac{N}{2}}\eta\bigg(\frac{x}{k}\bigg)\nu(x)
=\bigg(
k^{-\frac{N}{2}}\eta\bigg(\frac{x}{k}\bigg)\phi(x),
k^{-\frac{N}{2}}\eta\bigg(\frac{x}{k}\bigg)\psi(x)
\bigg)
=:(u_k,v_k),
\end{equation}
where $\eta \in C_0^{\infty}(\mathbb{R}^N)$ with $\eta \geq 0$ on $\mathbb{R}^N$ and
$$
\eta(x)= \begin{cases}1, & |x| \leq R_0, \\ 0, & |x| \geq R_0+1\end{cases}
$$
with $|x_0|<R_0$.

\begin{Pro}\label{Proposition 3.1}\rm{(\cite[Proposition 3.1]{Heinz-Kupper-Stuart1992JDE})}
Let $f:\R^N\rightarrow \mathbb{C}$ be a uniformly almost periodic function and let $g\in L^1(\R^N)$. Then 
$$\lim_{T\rightarrow\infty}\int_{\R^N}f(Tx)g(x)dx=M(f)\int_{\R^N}g(x)dx,$$
where $M(f)$ denotes the mean-value of the function $f$ and its definition is
$$
M(f)=\lim\limits_{T \rightarrow \infty} \frac{1}{T^N} \int_0^T\cdots\int_0^T f(x)dx.
$$
\end{Pro}

Next, we establish some estimates for $\{z_k\}$ that will
be used to construct a bounded sequence $\{\nu_n\}\subset W$;
see Lemma \ref{Lem3.4}.

\begin{Lem}\label{Lem3.2}
Let $\{z_k\}$ be defined by \eqref{eq3.15}, then there exist constant $C_0>0$ and $k_0 \in \mathbb{N}$ such that
$$
\Psi(z_k)\geq C_0 k^{-\frac{N(\alpha-2)}{2}-\tau}, \  \forall k>k_0 .
$$
\end{Lem}
\begin{proof}
Since $\nu$ is bounded,
\[
\|z_k\|_\infty\leq k^{-N/2}\|\eta\|_\infty\|\nu\|_\infty
\leq t_0
\]
for all sufficiently large $k$. Thus $(H_4)$ applies and gives
\begin{equation*}
\begin{split}
\Psi(z_k)&=\int_{\R^N}G(x,|z_k|)dx\geq L \int_{S}|x|^{-\tau}|z_k|^\alpha dx\\
&=Lk^{-\frac{N\alpha}{2}}\int_{S}|x|^{-\tau}\eta^\alpha\bigg(\frac{x}{k}\bigg)|\nu(x)|^\alpha dx\\
&=Lk^{-\frac{N(\alpha-2)}{2}-\tau}\int_{S_k}|y|^{-\tau}\eta^\alpha (y)|\nu(ky)|^\alpha dy,
\end{split}
\end{equation*}
where $S_k=\{y\in\R^N:ky\in S\}$. For $k\geq 1$, $S\subset S_k$ and define the characteristic function
$$\chi_S(y)= \begin{cases}1, & y\in S, \\ 0, & y\notin S,\end{cases}$$
we know that
\begin{equation*}
\Psi(z_k)\geq Lk^{-\frac{N(\alpha-2)}{2}-\tau}\int_{\R^N}|y|^{-\tau}\eta^\alpha(y)\chi_S(y)|\nu(ky)|^\alpha dy.
\end{equation*}
Noting that $|\nu|^\alpha$ is a uniformly almost periodic function, it follows from Proposition \ref{Proposition 3.1} that there exist constant $C_0>0$ and $k_0\in \mathbb{N}$ such that
\begin{equation*}
\Psi(z_k)\geq \frac{L}{2}k^{-\frac{N(\alpha-2)}{2}-\tau}M(|\nu|^\alpha)\int_{S}|y|^{-\tau}\eta^\alpha (y) dy=:C_0k^{-\frac{N(\alpha-2)}{2}-\tau}, \ \forall k>k_0.
\end{equation*}
\hfill
\end{proof}
\begin{Lem}\label{Lem3.3}
Let $\{z_k\}$ defined in \eqref{eq3.15}, then 
$$\lim\limits_{k \rightarrow \infty} \frac{((T-m I)z_k, z_k)_{2}}{\Psi(z_k)}=\lim \limits_{k \rightarrow \infty} \frac{\|(T-m I) z_k\|^2_{2}}{\Psi(z_k)}=0.$$
Moreover, there exist constants $C_1,C_2>0$ such that
 $$C_1\leq\|z_k\|\leq C_2$$ 
 for all sufficiently large \(k\).
\end{Lem}
\begin{proof}
For $k>0$, it is easy to know that
$$
\|z_k\|^2_{2}
=\int_{\mathbb{R}^N}(|u_k|^2+|v_k|^2)dx
=\int_{\mathbb{R}^N}
\eta^2(x)\big(|\phi(kx)|^2+|\psi(kx)|^2\big)dx
$$
and
\begin{equation*}
\begin{split}
\|\nabla z_k\|_{2}^2
&=\int_{\mathbb{R}^N}
\big(|\nabla u_k|^2+|\nabla v_k|^2\big)dx\\
&=k^{-N}\int_{\mathbb{R}^N}
\bigg|k^{-1}\nabla\eta\bigg(\frac{x}{k}\bigg)\phi
+\eta\bigg(\frac{x}{k}\bigg)\nabla\phi\bigg|^2dx+k^{-N}\int_{\mathbb{R}^N}
\bigg|k^{-1}\nabla\eta\bigg(\frac{x}{k}\bigg)\psi+\eta\bigg(\frac{x}{k}\bigg)\nabla\psi\bigg|^2dx\\
&\leq2k^{-2}\int_{\R^N}
|\nabla\eta(x)|^2
\big(|\phi(kx)|^2+|\psi(kx)|^2\big)dx+2\int_{\R^N}
\big(|\nabla\phi(kx)|^2+|\nabla\psi(kx)|^2\big)
\eta^2(x)dx.
\end{split}
\end{equation*}
Hence, we have
$$
\lim\limits_{k\rightarrow\infty}
\int_{\mathbb{R}^N}|z_k|^2dx
=M(|\phi|^2+|\psi|^2)
\int_{\mathbb{R}^N}|\eta(x)|^2dx>0,
$$
$$
\lim\limits_{k\rightarrow\infty}
\int_{\R^N}
|\nabla\eta(x)|^2
\big(|\phi(kx)|^2+|\psi(kx)|^2\big)dx
=M(|\phi|^2+|\psi|^2)
\int_{\R^N}|\nabla\eta|^2dx,
$$
and
$$
\lim\limits_{k\rightarrow\infty}
\int_{\R^N}
\big(|\nabla\phi(kx)|^2+|\nabla\psi(kx)|^2\big)
|\eta(x)|^2dx
=M(|\nabla\phi|^2+|\nabla\psi|^2)
\int_{\R^N}|\eta(x)|^2dx.
$$
By the above estimate, we know that there exist constants $C_1,C_2>0$ such that, for $k$ large enough
\begin{equation}\label{eq3.16}
0<C_1\leq\|z_k\|\leq C_2.
\end{equation}
Next, we prove the first conclusion. Using the fact that $T\nu=m \nu$, we have
\begin{equation}\label{eq3.17}
\begin{split}
(T-mI)z_k&=\left(\begin{array}{cc}
0 & -\Delta+V \\
-\Delta+V & 0
\end{array}\right)\left(\begin{array}{cc}
k^{-\frac{N}{2}}\displaystyle\eta\bigg(\frac{x}{k}\bigg)\phi \\
k^{-\frac{N}{2}}\displaystyle\eta\bigg(\frac{x}{k}\bigg)\psi
\end{array}\right)-\left(\begin{array}{cc}
m k^{-\frac{N}{2}}\displaystyle\eta\bigg(\frac{x}{k}\bigg)\phi \\
m k^{-\frac{N}{2}}\displaystyle\eta\bigg(\frac{x}{k}\bigg)\psi
\end{array}\right)\\
&=\left(\begin{array}{cc}
-k^{-\frac{N}{2}-2}\Delta\displaystyle\eta\bigg(\frac{x}{k}\bigg)\psi
-2k^{-\frac{N}{2}-1}\nabla\displaystyle\eta\bigg(\frac{x}{k}\bigg)\cdot\nabla\psi\\
-k^{-\frac{N}{2}-2}\Delta\displaystyle\eta\bigg(\frac{x}{k}\bigg)\phi
-2k^{-\frac{N}{2}-1}\nabla\displaystyle\eta\bigg(\frac{x}{k}\bigg)\cdot\nabla\phi
\end{array}\right).
\end{split}
\end{equation}
By integration by parts and \eqref{eq3.17} we know that
\begin{equation}\label{eq3.18}
\begin{split}
((T-mI)z_k,z_k)_{2}
&=-2k^{-N-2}\int_{\R^N}
\Delta\displaystyle\eta\bigg(\frac{x}{k}\bigg)
\displaystyle\eta\bigg(\frac{x}{k}\bigg)\phi\psi dx-2k^{-N-1}\int_{\R^N}
\nabla\displaystyle\eta\bigg(\frac{x}{k}\bigg)
\cdot\nabla(\phi\psi)
\displaystyle\eta\bigg(\frac{x}{k}\bigg)dx\\
&=2k^{-N-2}\int_{\R^N}
\bigg|\nabla\displaystyle\eta\bigg(\frac{x}{k}\bigg)\bigg|^2
\phi\psi dx\\
&=2k^{-2}\int_{\R^N}
|\nabla\displaystyle\eta(x)|^2\phi(kx)\psi(kx)dx.
\end{split}
\end{equation}
Thus, by \eqref{eq3.18}, one has
\begin{equation}\label{eq3.19}
\begin{split}
k^2|((T-mI)z_k,z_k)_{2}|
&=2\bigg|\int_{\R^N}
|\nabla\displaystyle\eta(x)|^2\phi(kx)\psi(kx)dx\bigg|\\
&\rightarrow2|M(\phi\psi)|
\int_{\R^N}|\nabla\eta|^2dx
\ \text{as}\ k\rightarrow\infty.
\end{split}
\end{equation}
Due to \eqref{eq3.17} and the fact that
$\phi^2$, $\psi^2$, $\phi\partial_i\phi$,
$\psi\partial_i\psi$, $\partial_i\phi\partial_j\phi$
and $\partial_i\psi\partial_j\psi$
are all uniformly almost periodic functions, we have
\begin{equation}\label{eq3.20}
\begin{split}
k^2\|(T-mI)z_k\|_{2}^2
=&k^2\int_{\R^N}
\bigg|-k^{-\frac{N}{2}-2}
\Delta\displaystyle\eta\bigg(\frac{x}{k}\bigg)\psi
-2k^{-\frac{N}{2}-1}
\nabla\displaystyle\eta\bigg(\frac{x}{k}\bigg)
\cdot\nabla\psi\bigg|^2dx\\
&+k^2\int_{\R^N}
\bigg|-k^{-\frac{N}{2}-2}
\Delta\displaystyle\eta\bigg(\frac{x}{k}\bigg)\phi
-2k^{-\frac{N}{2}-1}
\nabla\displaystyle\eta\bigg(\frac{x}{k}\bigg)
\cdot\nabla\phi\bigg|^2dx\\
=&k^{-2}\int_{\R^N}
(\Delta\eta(x))^2
\big(\phi^2(kx)+\psi^2(kx)\big)dx\\
&+4\int_{\R^N}\sum_{i,j=1}^N
\big(\partial_i\phi(kx)\partial_j\phi(kx)
+\partial_i\psi(kx)\partial_j\psi(kx)\big)
\partial_i\eta(x)\partial_j\eta(x)dx\\
&+4k^{-1}\int_{\R^N}\Delta\eta(x)
\sum_{i=1}^N
\big(\phi(kx)\partial_i\phi(kx)
+\psi(kx)\partial_i\psi(kx)\big)
\partial_i\eta(x)dx\\
\rightarrow&4\sum_{i,j=1}^N
M(\partial_i\phi\partial_j\phi
+\partial_i\psi\partial_j\psi)
\int_{\R^N}\partial_i\eta(x)\partial_j\eta(x)dx.
\end{split}
\end{equation}
Hence, by \eqref{eq3.19} and \eqref{eq3.20}, we get
\begin{equation}\label{eq3.22}
|((T- mI)z_k,z_k)_{2}|=O\bigg(\frac{1}{k^2}\bigg)
\end{equation}
and
\begin{equation}\label{eq3.23}
\|(T- mI)z_k\|^2_{2}=O\bigg(\frac{1}{k^2}\bigg).
\end{equation}
Hence, it follows from Lemma \ref{Lem3.2} that
\begin{equation*}
\frac{\|(T- mI)z_k\|_{2}^2}{\Psi(z_k)}\leq\frac{Ck^{-2}}{\Psi(z_k)}\leq Ck^{-2+\frac{N(\alpha-2)}{2}+\tau}\rightarrow0
\end{equation*}
and
\begin{equation*}
\frac{|((T-mI)z_k,z_k)_2|}{\Psi(z_k)}
\leq\frac{Ck^{-2}}{\Psi(z_k)}
\leq Ck^{-2+\frac{N(\alpha-2)}{2}+\tau}
\rightarrow0
\end{equation*}
as $k\rightarrow\infty$ since $\tau\in(0,\frac{4-N(\alpha-2)}{2})$. This completes the proof. 
\hfill
\end{proof}

Set $M:=\sup\{\Psi(z):z\in E \ \text{and}\ \|z\|=1 \}$. Obviously, it can be inferred from \eqref{eqg_4} that $M<\infty$. For any $z\in E$, we define $\rho(z):=\displaystyle\bigg(\frac{\Psi(z)}{2^q M}\bigg)^{\frac{1}{p}}$. We now establish an estimate for the nonlinear terms $\Psi(z)$. 
\begin{Lem}\label{Lem3.1}
Let $z\in E$ with $\|z\|=1$, then $0\leq \rho(z)\leq \frac{1}{2}$ and $\Psi(w)\geq 2^{-q}\Psi(z)$ for any $w\in E$ with $\|z-w\|\leq\rho(z)$.
\end{Lem}
\begin{proof}
By $(H_1)$ and the definition of $M$, it is easy to know that
\begin{equation}\label{eq3.12}
0\leq \rho(z)\leq \frac{1}{2}.
\end{equation}
$(H_1)$ and \eqref{eqg_1} imply that, for any $z,\eta\in E$,
$$\Psi^{\prime\prime}(z)[\eta,\eta]=\int_{\R^N}\bigg(g(x,|z|)|\eta|^2+g^\prime (x,|z|)\frac{(z\cdot \eta)^2} {|z|} \bigg)dx\geq0,$$
which implies that $\Psi$ is convex. Hence, for all $\mathbf{u},\mathbf{v}\in E$, we get
$$0\leq\Psi(\mathbf{u}+\mathbf{v})\leq\frac{1}{2}(\Psi(2\mathbf{u})+\Psi(2\mathbf{v}))\leq2^{q-1}(\Psi(\mathbf{u})+\Psi(\mathbf{v})),$$
and so
\begin{equation}\label{eq3.13}
\Psi(\mathbf{u})\geq2^{1-q}\Psi(\mathbf{u}+\mathbf{v})-\Psi(\mathbf{v}).
\end{equation}
For $\mathbf{u},z\in E$ with $\|\mathbf{u}\|=1$ and if $0<\|z\|\leq\rho(\mathbf{u})\leq\frac{1}{2}$, from \eqref{eqg_3} we have
$$0\leq\Psi(z)=\int_{\R^N}G(x,|z|)dx\leq \|z\|^p\int_{\R^N}G\bigg(x,\frac{|z|}{\|z\|}\bigg)dx\leq\rho^p(\mathbf{u})M=2^{-q}\Psi(\mathbf{u}).$$
By \eqref{eq3.13}, for all $\mathbf{u},z\in E$ with $\|\mathbf{u}\|=1$ and $\|z\|\leq\rho(\mathbf{u})$, we know that
$$\Psi(\mathbf{u}-z)\geq2^{1-q}\Psi(\mathbf{u})-\Psi(z)\geq2^{-q}\Psi(\mathbf{u}).$$
Thus, the proof is completed.
\hfill
\end{proof}

\begin{Lem}\label{Lem3.4}
There exists a bounded sequence $\{\nu_n\}\subset W$ such that
\begin{equation*}
\lim_{n\to\infty}
\frac{((T-mI)\nu_n,\nu_n)_2}{\Psi(\nu_n)}=0
\end{equation*}
and
\begin{equation*}
\|\nu_n\|_2=1,\qquad
\Psi(\nu_n)>0,\qquad n\in\mathbb N.
\end{equation*}
\end{Lem}

\begin{proof}
Let $\{z_k\}$ be defined by \eqref{eq3.15}. Then
\begin{equation}\label{eq3.24}
\|z_k^-\|^2+m\|z_k^-\|_2^2
=
\bigl|((T-mI)z_k,z_k^-)_2\bigr|
\le
\|(T-mI)z_k\|_2\|z_k^-\|_2.
\end{equation}
By \eqref{eq3.23} and \eqref{eq3.24},
\begin{equation}\label{eq3.25}
\|z_k^-\|\to0,
\qquad
\|z_k^-\|_2\to0
\quad\text{as }k\to\infty.
\end{equation}
\eqref{eq3.16} and \eqref{eq3.25} imply that
there exist constants $C_3,C_4>0$ such that, for all
sufficiently large $k$,
\begin{equation}\label{eq3.26}
C_3\le\|z_k^+\|_2\le C_4.
\end{equation}
Set $\nu_k:=z_k^+/\|z_k^+\|_2$.
Then $\nu_k\in E^+$ and $\|\nu_k\|_2=1$.
By \eqref{eq3.24}, \eqref{eq3.22}, and \eqref{eq3.23},
\begin{equation}\label{eq+}
\begin{aligned}
0\le ((T-mI)z_k^+,z_k^+)_2
&=((T-mI)z_k,z_k)_2
  +\|z_k^-\|^2+m\|z_k^-\|_2^2\\
&\le ((T-mI)z_k,z_k)_2
  +\frac{1}{m}\|(T-mI)z_k\|_2^2\\
&=O(k^{-2}).
\end{aligned}
\end{equation}
Combining this estimate with \eqref{eq3.26}, we obtain
\[
\|\nu_k\|^2
=\frac{\|z_k^+\|^2}{\|z_k^+\|_2^2}
=m+\frac{((T-mI)z_k^+,z_k^+)_2}{\|z_k^+\|_2^2}
\longrightarrow m.
\]
Thus, $\|\nu_k\|\to\sqrt m$.
In particular, $\{\nu_k\}$ is bounded in $E$,
and $\nu_k\in W$ for all sufficiently large $k$.
Moreover, by \eqref{eq3.16}, \eqref{eq3.24}
and Lemma \ref{Lem3.3}, for $k$ large enough,
\[
\left\|\frac{z_k^-}{\|z_k\|}\right\|
\le \frac{\|z_k^-\|}{C_1}
\le \frac{\|(T-mI)z_k\|_2}{C_1\sqrt m}
\le \rho\left(\frac{z_k}{\|z_k\|}\right).
\]
Then Lemma \ref{Lem3.1} gives
\[
\Psi\left(\frac{z_k^+}{\|z_k\|}\right)
\ge 2^{-q}
\Psi\left(\frac{z_k}{\|z_k\|}\right).
\]
Together with \eqref{eq3.16} and \eqref{eqg_3},
this yields
\begin{equation}\label{eq3.27}
\Psi(z_k^+)\ge C\Psi(z_k).
\end{equation}
By \eqref{eq+} and \eqref{eq3.26}, we have
\begin{equation}\label{eq3.28}
\begin{aligned}
0\le ((T-mI)\nu_k,\nu_k)_2
&=\frac{((T-mI)z_k^+,z_k^+)_2}{\|z_k^+\|_2^2}\\
&\le \frac{1}{C_3^2}
\left(
((T-mI)z_k,z_k)_2
+\frac{1}{m}\|(T-mI)z_k\|_2^2
\right).
\end{aligned}
\end{equation}
Furthermore, \eqref{eq3.26}, \eqref{eq3.27},
and \eqref{eqg_3} imply that
\[
\Psi(\nu_k)
=\Psi\left(\frac{z_k^+}{\|z_k^+\|_2}\right)
\ge C\Psi(z_k^+)
\ge C\Psi(z_k)>0
\]
for all sufficiently large $k$.
Consequently, by \eqref{eq3.28} and Lemma \ref{Lem3.3},
\[
0\le
\frac{((T-mI)\nu_k,\nu_k)_2}{\Psi(\nu_k)}
\le
C\frac{
((T-mI)z_k,z_k)_2
+\frac{1}{m}\|(T-mI)z_k\|_2^2
}{\Psi(z_k)}
\longrightarrow0
\quad\text{as }k\to\infty.
\]
This completed the proof.
\end{proof}

\subsection{The constrained reduction and the multiplier map}
Given $a>0$, consider the following problem:
$$c_a=\inf_{w\in W\cap S_a}\sup_{z\in M(w)}J_0(z),$$
where $w\in E^+$, $S_a=\{z\in E\mid\|z\|_2=a\}$. Our aim is to reduce \(J_0\) on \(W\). We show that, for every \(w\in W\) with sufficiently small \(\|w\|_2\), \(J_0\) has a unique maximizer on \(M(w)\), denoted by \(\Phi(w)\).

\begin{Lem}\label{Lem3.5}
For any $w\in W$ with $\|w\|_{2}$ small enough, there exists
a unique $\Phi(w)\in M(w)$ such that
\begin{equation*}
J_0(\Phi(w))=\max_{z\in M(w)}J_0(z).
\end{equation*}
Moreover, $\Phi$ is a continuously differentiable function.
\end{Lem}
\begin{proof}
Given $w\in W$, we denote
\begin{equation*}
  \Xi(w)=\left\{\phi\in E^-\mid\|\phi\|\leq \frac{\sqrt{m}\|w\|_{2}}{2}\right\}.
\end{equation*}
Our first goal is to construct a surjection $h(w,\cdot): \Xi(w)\rightarrow M(w)$, which show that
\begin{equation}\label{eq3.3}
\sup_{z\in M(w)} J_0 (z)=\sup_{\phi\in \Xi(w)} J_0 (h(w,\phi)).
\end{equation}
In fact, define the map $h:\mathscr{D}(h)\to E$ by
\begin{equation}\label{eq3.2}
h(w,\phi)=\sqrt{\|w\|_{2}^2-\|\phi\|_{2}^2}\frac{w}{\|w\|_{2}}+\phi,
\end{equation}
where  \begin{equation*}
        \mathscr{D}(h)=\left\{(w,\phi)\in W\times E^-\mid\|\phi\|\leq \frac{\sqrt{m}\|w\|_{2}}{2}\right\}.
    \end{equation*}
Based on the fact that
$$\displaystyle\|\phi\|_{2}\leq
\frac{1}{\sqrt{m}}\|\phi\|\leq  \frac{\|w\|_{2}}{2} \ \ \mbox{if}\  (w,\phi)\in \mathscr{D}(h),$$
we know that $h(w,\phi)$ is well-defined.

For any $\eta\in E^-$,
we have
\begin{equation}\label{eq3.4}
\frac{\partial h}{\partial \phi}(w,\phi)\eta
=-\frac{(\phi,\eta)_{2}}{\|w\|_{2}\sqrt{\|w\|_{2}^2-\|\phi\|_{2}^2}}w+\eta.
\end{equation}
Therefore, a direct computation jointly with \eqref{eq3.4}, shows that
\begin{equation*}
    \begin{split}
    \left\langle \frac{\partial (J_0\circ h)}{\partial \phi}(w,\phi),\eta \right\rangle
        &= \left\langle J_0^\prime(h(w,\phi)),\frac{\partial h}{\partial \phi}(w,\phi)\eta\right\rangle\\
        &
        =-\langle\phi,\eta \rangle-\frac{\|w\|^2}{\|w\|^2_{2}}(\phi,\eta)_{2}-\int_{\R^N}g(x,|h(w,\phi)|)h(w,\phi)\cdot \eta dx\\
        &\quad+\frac{(\phi,\eta)_{2}}{\|w\|_2\sqrt{\|w\|_{2}^2-\|\phi\|_{2}^2}}\int_{\R^N}g(x,|h(w,\phi)|)h(w,\phi)\cdot w dx,
    \end{split}
\end{equation*}
and
\begin{equation}\label{eq3.5}
\frac{\partial^2(J_0\circ h)}{\partial\phi^2}
(w,\phi)[\eta,\eta]
=-\|\eta\|^2
-\frac{\|w\|^2\|\eta\|_2^2}{\|w\|_2^2}
+\sum_{i=1}^4 I_i,
\end{equation}
where
\begin{equation*}
\begin{split}
I_1
=-\int_{\R^N}
\frac{g'(x,|h(w,\phi)|)}{|h(w,\phi)|}
\bigg(
h(w,\phi)\cdot\eta
-\frac{(\phi,\eta)_2}
{\|w\|_2\sqrt{\|w\|_2^2-\|\phi\|_2^2}}
h(w,\phi)\cdot w
\bigg)^2\,dx,
\end{split}
\end{equation*}

\begin{equation*}
I_2
=-\int_{\R^N}g(x,|h(w,\phi)|)
\left|
\eta-\frac{(\phi,\eta)_2}
{\|w\|_2\sqrt{\|w\|_2^2-\|\phi\|_2^2}}w
\right|^2\,dx,
\end{equation*}

\begin{equation*}
I_3
=\frac{\|\eta\|_2^2}
{\|w\|_2\sqrt{\|w\|_2^2-\|\phi\|_2^2}}
\int_{\R^N}
g(x,|h(w,\phi)|)h(w,\phi)\cdot w\,dx
\end{equation*}
and
\begin{equation*}
I_4
=\frac{(\phi,\eta)_2^2}
{\|w\|_2(\|w\|_2^2-\|\phi\|_2^2)^{3/2}}
\int_{\R^N}
g(x,|h(w,\phi)|)h(w,\phi)\cdot w\,dx.
\end{equation*}

Using the definition of $h$ and the equivalence
of norms on $W$, we can check that
\begin{equation}\label{eq3.6}
\|h(w,\phi)\|\leq C\|w\|_2.
\end{equation}
Since $\phi\in\Xi(w)$, we have
\begin{equation*}
\sqrt{\|w\|_2^2-\|\phi\|_2^2}
\geq\frac{\sqrt{3}}{2}\|w\|_2
\end{equation*}
and
\begin{equation*}
\begin{split}
\left\|
\eta-\frac{(\phi,\eta)_2}
{\|w\|_2\sqrt{\|w\|_2^2-\|\phi\|_2^2}}w
\right\|
&\leq\|\eta\|
+\frac{\|\phi\|_2\|\eta\|_2\|w\|}
{\|w\|_2\sqrt{\|w\|_2^2-\|\phi\|_2^2}}\\
&\leq C\|\eta\|.
\end{split}
\end{equation*}
Consequently, by \eqref{eqg_1}, \eqref{eqg_4},
\eqref{eq3.6} and H\"older's inequality, we obtain
\begin{equation*}
\begin{split}
|I_1|+|I_2|
&\leq C\bigl(
\|h(w,\phi)\|_p^{p-2}
+\|h(w,\phi)\|_q^{q-2}
\bigr)
\left\|
\eta-\frac{(\phi,\eta)_2}
{\|w\|_2\sqrt{\|w\|_2^2-\|\phi\|_2^2}}w
\right\|^2\\
&\leq C\bigl(
\|w\|_2^{p-2}+\|w\|_2^{q-2}
\bigr)\|\eta\|^2.
\end{split}
\end{equation*}
Moreover,
\begin{equation*}
\begin{split}
&\frac{\|\eta\|_2^2}
{\|w\|_2\sqrt{\|w\|_2^2-\|\phi\|_2^2}}
+\frac{(\phi,\eta)_2^2}
{\|w\|_2(\|w\|_2^2-\|\phi\|_2^2)^{3/2}}\leq
C\frac{\|\eta\|^2}{\|w\|_2^2},
\end{split}
\end{equation*}
and hence
\begin{equation*}
\begin{split}
|I_3|+|I_4|
&\leq C\frac{\|\eta\|^2}{\|w\|_2^2}
\bigl(
\|h(w,\phi)\|_p^{p-1}\|w\|_p
+\|h(w,\phi)\|_q^{q-1}\|w\|_q
\bigr)\\
&\leq C\bigl(
\|w\|_2^{p-2}+\|w\|_2^{q-2}
\bigr)\|\eta\|^2.
\end{split}
\end{equation*}
Therefore, for $\|w\|_2$ small enough,
\begin{equation}\label{eq3.8}
|I_i|\leq\frac16\|\eta\|^2,
\qquad i=1,2,3,4.
\end{equation}
Hence, due to \eqref{eq3.5} and \eqref{eq3.8},
we infer that
\begin{equation*}
\frac{\partial^2(J_0\circ h)}{\partial\phi^2}
(w,\phi)[\eta,\eta]
\leq-\frac16\|\eta\|^2,
\end{equation*}
which means that the functional
$\phi\mapsto(J_0\circ h)(w,\phi)$ is strictly
concave on $\Xi(w)$ for $\|w\|_2$ small enough.
Since $\Xi(w)$ is weakly compact and
$\phi\mapsto(J_0\circ h)(w,\phi)$ is continuous
and concave on $\Xi(w)$, this functional is
weakly upper semicontinuous and attains its
maximum on $\Xi(w)$.

On the other hand, for any $w\in W$
with $\|w\|_{2}$ small enough, we assume that there exist the $\phi_1$ such that $\displaystyle\|\phi_1\|=\frac{\sqrt{m} \|w\|_{2}}{2}$ and $(J_0\circ h)(w,\phi_1)=\sup\limits_{\phi\in \Xi(w)} J_0 (h(w,\phi))$, then
\begin{equation*}
\begin{split}
&\quad (J_0\circ h)(w,0)-(J_0\circ h)(w,\phi_1)\\
&
=\frac{\|\phi_1\|_{2}^2\|w\|^2}{2\|w\|_{2}^2}+\frac{1}{2}\|\phi_1\|^2-\int_{\mathbb{R}^N}G(x,|w|)dx+\int_{\mathbb{R}^N}G(x,|h(w,\phi_1)|)dx\\
&\geq \frac{m}{8}\|w\|_{2}^2-C\|w\|_{p}^p-C\|w\|_{q}^q\\
&\geq \frac{m}{8}\|w\|_{2}^2-C\|w\|_{2}^p\\
&\geq \frac{m}{9}\|w\|_{2}^2,
\end{split}
\end{equation*}
since $p>2$, and therefore
$$(J_0\circ h)(w,\phi_1)< (J_0\circ h)(w,0)\leq\sup\limits_{\phi\in \Xi(w)} J_0 (h(w,\phi)),$$
which is a contradiction and $(J_0\circ h)(w,\cdot)$ cannot
achieve its maximum on the boundary of $\Xi(w)$. 
By strict concavity, there exists a unique 
$\phi(w)\in \Xi(w)$ such that
\begin{equation*}
    J_0(h(w,\phi(w)))=\sup_{\tilde{\phi}\in \Xi(w)}J_0(h(w,\tilde{\phi})).
\end{equation*}
Setting 
\begin{equation}\label{eq3.9}
   \Phi(w):=h(w,\phi(w)),
\end{equation}
then $\Phi(w)\in M(w)$ satisfies
\begin{equation*}
    J_0(\Phi(w))=\sup_{z\in M(w)}J_0(z).
\end{equation*}
Moreover, we have $\displaystyle \frac{\partial (J_0\circ h)}{\partial \phi}(w,\phi(w))=0$. From the implicit function theorem, we know that $\phi(w)$ is continuously differentiable, and so $\Phi$ is. 
\end{proof}

Following \cite[Theorem 2.1]{Buffoni-Jeanjean1993},
for $z\in E$ with $z^+\neq0$, we define
the multiplier map $\kappa$ by
    \begin{equation*}
\kappa(z)=\frac{\|z^+\|^2-\displaystyle\int_{\mathbb{R}^N} g(x,|z|)z\cdot z^+dx}{\|z^+\|_{2}^2}
\end{equation*}
and consider the operator $H$ given by
    \begin{equation}\label{eq3.10}
Hz:=Tz-g(x,|z|)z-\kappa(z)z.
\end{equation}
By the definition of $\Phi$, for any $\eta\in E^-$, it is not difficult to see that
\begin{equation}\label{eq3.11}
\left\langle H\circ \Phi(w),\eta\right\rangle_{E^*,E}=0.
\end{equation}
Fix $a_*>0$ sufficiently small so that
Lemma~\ref{Lem3.5} applies whenever $\|w\|_2<a_*$,
and set
$W_*:=\{w\in W:\|w\|_2<a_*\}$.
We now consider the reduced functional
$I:W_*\to\mathbb{R}$ defined by
\begin{equation*}
    I(w)=(J_0 \circ \Phi)(w).
\end{equation*}
Using Lemma \ref{Lem3.5}, we know that the functional $I$ is of class $C^1$ and
\begin{equation*}
    c_a=\inf_{w\in W\cap S_a}I(w).
\end{equation*}

Next, we will use the sequence $\nu_n$ obtained in Lemma \ref{Lem3.4} to prove the estimation of $c_a$.

\begin{Lem}\label{Lem3.6}
For the prescribed constant $a$ small enough, there holds $$c_a<\frac{ma^2}{2}.$$ 
\end{Lem}
\begin{proof}
    Let $\hat{\nu}_n=a\nu_n$, where $\nu_n$ is defined in Lemma \ref{Lem3.4}. By the definition of $\Xi(w)$, we know that 
\begin{equation*}
\frac{\sqrt{3}}{2}\|\hat{\nu}_n\|_{2}\leq\sqrt{\|\hat{\nu}_n\|_{2}^2-\|\phi(\hat{\nu}_n)\|_{2}^2}\leq \|\hat{\nu}_n\|_{2},
\end{equation*}
and therefore
\begin{equation}\label{eq5.21}
0<\displaystyle\frac{\sqrt{\|\hat{\nu}_n\|_{2}^2-\|\phi(\hat{\nu}_n)\|_{2}^2}}{2\|\hat{\nu}_n\|_{2}}\leq\frac{1}{2}.
\end{equation}
Using \eqref{eqg_3}, \eqref{eqg_4}, \eqref{eq3.9}, \eqref{eq5.21} and the convexity of $\Psi$, we may estimate
    \begin{equation*}
        \begin{split}
            I(\hat{\nu}_n)&=\frac{a^2-\|\phi(\hat{\nu}_n)\|_{2}^2}{2a^2}\|\hat{\nu}_n\|^2-\frac{1}{2}\|\phi(\hat{\nu}_n)\|^2-\Psi(\Phi(\hat{\nu}_n))\\
            &\leq \frac{1}{2}\|\hat{\nu}_n\|^2-\frac{1}{2}\|\phi(\hat{\nu}_n)\|^2-2\Psi\bigg(\frac{\sqrt{a^2-\|\phi(\hat{\nu}_n)\|_2^2}}{2a}\hat{\nu}_n\bigg)+\Psi(-\phi(\hat{\nu}_n))\\
            &\leq \frac{1}{2}\|\hat{\nu}_n\|^2-\frac{1}{2}\|\phi(\hat{\nu}_n)\|^2-2^{1-2q}\Psi(\hat{\nu}_n)+\Psi(-\phi(\hat{\nu}_n))\\
            &\leq \frac{1}{2}\|\hat{\nu}_n\|^2-\frac{1}{2}\|\phi(\hat{\nu}_n)\|^2- 2^{1-2q}\Psi(\hat{\nu}_n)+C\|\phi(\hat{\nu}_n)\|^p+C\|\phi(\hat{\nu}_n)\|^q\\
            &\leq \frac{a^2}{2}\|\nu_n\|^2-2^{1-2q}a^q\Psi(\nu_n)-\frac{1}{2}\|\phi(\hat{\nu}_n)\|^2+C\|\phi(\hat{\nu}_n)\|^p+C\|\phi(\hat{\nu}_n)\|^q.
        \end{split}
    \end{equation*}
Since $\phi(\hat{\nu}_n)\in\Xi(\hat{\nu}_n)$ and
$\|\hat{\nu}_n\|_2=a$, we have
$\|\phi(\hat{\nu}_n)\|\leq\sqrt{m}\,a/2$.
Consequently, for $a$ small enough,
    \begin{equation*}
        -\frac{1}{2}\|\phi(\hat{\nu}_n)\|^2+C\|\phi(\hat{\nu}_n)\|^p+C\|\phi(\hat{\nu}_n)\|^q\leq 0.
\end{equation*}
Also, we recall the following estimate obtained in Lemma \ref{Lem3.4}:
\begin{equation*}
        \frac{\|\nu_n\|^2-m}{\Psi(\nu_n)}<2^{2-2q}a^{q-2}\  \text{as}\ n\rightarrow\infty.
    \end{equation*}
   Thus,
    \begin{equation*}
        c_a\leq I(\hat{\nu}_n)\leq  \frac{a^2}{2}\|\nu_n\|^2-2^{1-2q}a^q\Psi(\nu_n)
        <\frac{ma^2}{2},
    \end{equation*}
and this concludes the proof. \end{proof}

If the minimizing sequence $\{w_n\}$ of $I$ restricted to $W\cap S_a$, i.e.  $\{w_n\} \subset W\cap S_a$ satisfying $I(w_n)\rightarrow c_a$ as $n\rightarrow\infty$ for $a>0$ small enough, we
can assume the $\{w_n\}$ belongs to the following set (see \cite[Theorem 2.2]{Buffoni-Jeanjean1993}
):
\begin{equation*}
    X_a=\left\{w\in E^+\mid\|w\|_{2}=a~\text{and}~\|w\|^2\leq \left(m+a^{\frac{p-2}{2}}\right)a^2\right\}\subset W \cap S_a.
\end{equation*}
For $w\in W\cap S_a$, we have
\begin{equation*}
I(w)\geq J_0(w)\geq\frac12\|w\|^2-Ca^p.
\end{equation*}
Hence, if
$\|w\|^2\geq(m+a^{(p-2)/2})a^2$, then
\begin{equation*}
I(w)\geq\frac{ma^2}{2}
+\frac12a^{(p+2)/2}-Ca^p
>\frac{ma^2}{2}
\end{equation*}
for $a$ small enough.
Since $c_a<ma^2/2$, every minimizing sequence
eventually lies in the relative interior of $X_a$
in $E^+\cap S_a$.
Here $I_{|X_a}'(w)$ denotes the derivative of $I$
along the sphere $E^+\cap S_a$, evaluated at
$w\in X_a$. Thus, for all $w\in X_a$,
\begin{equation*}
\left\langle I_{|X_a}^{\prime}\left(w\right),
\eta\right\rangle=\left\langle I^{\prime}\left(w\right),
\eta\right\rangle~~\text{for}~~\eta\in T_w,
\end{equation*}
where
\begin{equation*}
T_w=\left\{\eta\in E^+\mid (w,\eta)_{2}=0\right\}.
\end{equation*} 
\subsection{Proof of Theorem \ref{Thm1.1}}
To begin with, we recall that $\left\{w_n\right\} \subset X_a$ is termed a Palais-Smale sequence at level $c$, or a $(PS)_c$-sequence, for $I$ if $I\left(w_n\right) \rightarrow
c$
 and $I_{\mid X_{a}}^{\prime}(w_n) \rightarrow 0$ as $n\rightarrow\infty$. 
We say that $I_{\mid X_a}$ satisfies $(PS)_c$ if every $(PS)_c$-sequence in $X_a$ has a subsequence converging in $E$. Our first goal will be to show that the functional $I$ satisfies the Palais-Smale condition restricted to $X_a$  at all levels strictly below $\displaystyle\frac{ma^2}{2}$.
\begin{Lem}\label{Lem3.8}
There exists $a_0>0$ such that, for every $a\in(0,a_0)$,
the functional $I_{|X_a}$ satisfies the Palais--Smale
condition at every level $c<\displaystyle\frac{ma^2}{2}$.
\end{Lem}
\begin{proof}
   Let $a_0\in(0,1)$ be sufficiently small so that the conclusions of Lemma~\ref{Lem3.5} hold, and fix $a\in(0,a_0)$.
Let $\{w_n\}\subset X_a$ satisfy
\begin{equation}\label{eq3-ps-sequence}
I(w_n)\to c<\frac{ma^2}{2},
\qquad
I_{|X_a}'(w_n)\to0.
\end{equation}

\medskip
\noindent Step 1. $\{\Phi(w_n)\}\subset E$ and
$\{\kappa(\Phi(w_n))\}\subset\mathbb R$ are bounded.
\newline
Since $\{w_n\}\subset X_a$, we have
\begin{equation}\label{eq_bound-1}
\|w_n\|_2=a,
\qquad
\|w_n\|^2\leq
\left(m+a^{\frac{p-2}{2}}\right)a^2.
\end{equation}
Moreover, $\phi(w_n)\in\Xi(w_n)$ and \eqref{eq2.5} yield
\begin{equation}\label{eq_bound-2}
\|\phi(w_n)\|^2\leq\frac{m}{4}a^2,
\qquad
\|\phi(w_n)\|_2^2
\leq\frac{1}{m}\|\phi(w_n)\|^2
\leq\frac{a^2}{4}.
\end{equation}
By \eqref{eq_bound-1} and \eqref{eq_bound-2}, we obtain
\[
\begin{aligned}
\|\Phi(w_n)\|^2
&=
\frac{a^2-\|\phi(w_n)\|_2^2}{a^2}\|w_n\|^2
+\|\phi(w_n)\|^2\\
&\leq\|w_n\|^2+\|\phi(w_n)\|^2\\
&\leq
\left(\frac{5m}{4}+a^{\frac{p-2}{2}}\right)a^2
\leq\frac{5m+4}{4}a^2.
\end{aligned}
\]
Thus, $\{\Phi(w_n)\}$ is bounded in $E$. In addition,
\begin{equation}\label{eq3-mass-bound}
\|\Phi(w_n)\|_2^2=a^2,
\qquad
\|\Phi^+(w_n)\|_2^2
=a^2-\|\phi(w_n)\|_2^2
\geq\frac34a^2.
\end{equation}
By \eqref{eqg_2}, \eqref{eqg_4} and H\"older's inequality, we have
\begin{equation}\label{eq3-nonlinear-bound}
\begin{aligned}
&\left|
\int_{\mathbb R^N}
g(x,|\Phi(w_n)|)\Phi(w_n)\cdot\Phi^+(w_n)\,dx
\right|\\
&\quad\leq
C\left(
\|\Phi(w_n)\|_p^{p-1}\|\Phi^+(w_n)\|_p
+\|\Phi(w_n)\|_q^{q-1}\|\Phi^+(w_n)\|_q
\right)\\
&\quad\leq
C\left(\|\Phi(w_n)\|^p+\|\Phi(w_n)\|^q\right)\\
&\quad\leq C(a^p+a^q)\leq Ca^p.
\end{aligned}
\end{equation}
Since
\[
\Phi^+(w_n)
=\frac{\sqrt{a^2-\|\phi(w_n)\|_2^2}}{a}\,w_n,
\]
it follows from \eqref{eq2.5} and the definition of
$X_a$ that
\begin{equation}\label{eq3-quotient-bound}
m\leq
\frac{\|\Phi^+(w_n)\|^2}{\|\Phi^+(w_n)\|_2^2}
=\frac{\|w_n\|^2}{a^2}
\leq m+a^{\frac{p-2}{2}}.
\end{equation}
Therefore, by the definition of $\kappa$,
\eqref{eq3-mass-bound}, \eqref{eq3-nonlinear-bound},
and \eqref{eq3-quotient-bound}, we obtain
\[
\begin{aligned}
\left|
\kappa(\Phi(w_n))-\frac{\|w_n\|^2}{a^2}
\right|
&=
\frac{
\left|\displaystyle\int_{\mathbb R^N}
g(x,|\Phi(w_n)|)\Phi(w_n)\cdot\Phi^+(w_n)\,dx\right|
}{
\|\Phi^+(w_n)\|_2^2
}\leq Ca^{p-2}.
\end{aligned}
\]
Consequently,
\begin{equation}\label{eq3-kappa-bounds}
m-Ca^{p-2}
\leq\kappa(\Phi(w_n))
\leq m+a^{\frac{p-2}{2}}+Ca^{p-2},
\end{equation}
which proves that $\{\kappa(\Phi(w_n))\}$ is bounded.

\medskip
\noindent Step 2. $\Phi(w_n)\rightharpoonup z_a$ in $E$ and the weak limit $z_a$ satisfies \eqref{eq3.35}.
\newline
By Step 1 and \eqref{eq3-ps-sequence}, passing to a subsequence, we may assume that
\begin{equation}\label{eq3-weak-convergence}
\Phi(w_n)\rightharpoonup z_a\text{ in }E
\quad\text{and}\quad
\kappa(\Phi(w_n))\to\omega_a\text{ in }\mathbb R
\quad\text{as }n\to\infty,
\end{equation}
\begin{equation}\label{eq3-local-convergence}
\begin{aligned}
\Phi(w_n)&\to z_a
&&\text{in }L^r_{\mathrm{loc}}(\mathbb R^N,\mathbb R^2),
\quad 2\leq r<2^*,\\
\Phi(w_n)(x)&\to z_a(x)
&&\text{a.e. in }\mathbb R^N,
\end{aligned}
\end{equation}
and
\begin{equation*}
\left|\left\langle I^\prime(w_n),\eta\right\rangle\right|
\leq\frac{1}{n}\|\eta\|,
\qquad \forall\,\eta\in T_{w_n}.
\end{equation*}
We first observe that
\[
(\Phi(w_n),\Phi^\prime(w_n)\eta)_2
=(w_n,\eta)_2=0,
\qquad \forall\,\eta\in T_{w_n}.
\]
Thus,
\begin{equation}\label{eq3.30}
\begin{split}
\langle I^\prime(w_n),\eta\rangle
&=\langle J_0^\prime(\Phi(w_n)),
          \Phi^\prime(w_n)\eta\rangle
-\kappa(\Phi(w_n))
       (\Phi(w_n),\Phi^\prime(w_n)\eta)_2\\
&=\left\langle H\circ\Phi(w_n),
               \Phi^\prime(w_n)\eta\right\rangle.
\end{split}
\end{equation}
Also, a direct computation yields, for $\eta\in T_{w_n}$,
\begin{equation}\label{eq3.31}
\begin{split}
\Phi^\prime(w_n)\eta
&=\frac{\sqrt{a^2-\|\phi(w_n)\|_2^2}}{a}\eta
-\frac{(\phi(w_n),\phi^\prime(w_n)\eta)_2}
{a\sqrt{a^2-\|\phi(w_n)\|_2^2}}w_n
+\phi^\prime(w_n)\eta.
\end{split}
\end{equation}
By the definition of $H$, we have
\begin{equation}\label{eq3.32}
\left\langle H\circ\Phi(w_n),w_n\right\rangle=0.
\end{equation}
Consequently, \eqref{eq3.30}--\eqref{eq3.32}
and \eqref{eq3.11} give
\begin{equation}\label{eq3.29}
\langle I^\prime(w_n),\eta\rangle
=\frac{\sqrt{a^2-\|\phi(w_n)\|_2^2}}{a}
 \left\langle H\circ\Phi(w_n),\eta\right\rangle,
\qquad \forall\,\eta\in T_{w_n}.
\end{equation}
Moreover, for any $\eta\in E^+$, we have
\[
\eta-\frac{(w_n,\eta)_2}{a^2}w_n\in T_{w_n}.
\]
Using \eqref{eq3.32}, \eqref{eq3.29}, and the bounds
established in Step 1, we obtain
\[
\begin{aligned}
\left|\left\langle H\circ\Phi(w_n),\eta\right\rangle\right|&=
\left|\left\langle
H\circ\Phi(w_n),
\eta-\frac{(w_n,\eta)_2}{a^2}w_n
\right\rangle\right|\\
&=
\frac{a}{\sqrt{a^2-\|\phi(w_n)\|_2^2}}
\left|\left\langle
I^\prime(w_n),
\eta-\frac{(w_n,\eta)_2}{a^2}w_n
\right\rangle\right|\\
&\leq
\frac{a}{n\sqrt{a^2-\|\phi(w_n)\|_2^2}}
\left(
\|\eta\|+\frac{|(w_n,\eta)_2|}{a^2}\|w_n\|
\right)\\
&\leq\frac{C}{n}\|\eta\|.
\end{aligned}
\]
Together with \eqref{eq3.11}, this proves that
\begin{equation}\label{eq3.33}
H\circ\Phi(w_n)\to0
\quad\text{in }H^{-1}(\mathbb R^N,\mathbb R^2).
\end{equation}
By the definition of $H$ and \eqref{eq3.33}, we have
\begin{equation}\label{eq3.34}
T\circ\Phi(w_n)
=g(x,|\Phi(w_n)|)\Phi(w_n)
+\kappa(\Phi(w_n))\Phi(w_n)+o_n(1)
\quad\text{in }E^*.
\end{equation}

By $(H_1)$--$(H_3)$ and \eqref{eq3-local-convergence},
for every $\eta\in C_c^\infty(\mathbb R^N,\mathbb R^2)$,
we have
\begin{equation}\label{eq3-local-nonlinear-limit}
\int_{\mathbb R^N}
\bigl(g(x,|\Phi(w_n)|)\Phi(w_n)-g(x,|z_a|)z_a\bigr)
\cdot\eta\,dx
\to0
\quad\text{as }n\to\infty.
\end{equation}
Combining \eqref{eq3.34},
\eqref{eq3-weak-convergence},
and \eqref{eq3-local-nonlinear-limit},
and using the density of
$C_c^\infty(\mathbb R^N,\mathbb R^2)$ in $E$, we obtain
\begin{equation}\label{eq3.35}
Tz_a=g(x,|z_a|)z_a+\omega_a z_a
\quad\text{in }E^*.
\end{equation}

\medskip
\noindent Step 3. $\Phi(w_n)\to z_a$ in $E$.
\newline
By \eqref{eqg_2}, \eqref{eq3.33}, and \eqref{eq3-mass-bound},
\begin{equation}\label{eq3-multiplier-energy}
\begin{aligned}
\kappa(\Phi(w_n))a^2
&=2I(w_n)
+2\int_{\mathbb R^N}G(x,|\Phi(w_n)|)\,dx-\int_{\mathbb R^N}g(x,|\Phi(w_n)|)|\Phi(w_n)|^2\,dx+o_n(1)\\
&\leq2I(w_n)+o_n(1).
\end{aligned}
\end{equation}
By \eqref{eq3-kappa-bounds},
\eqref{eq3-weak-convergence}, \eqref{eq3-ps-sequence},
and \eqref{eq3-multiplier-energy},
taking $a_0$ smaller if necessary so that
$Ca_0^{p-2}<m$, we obtain
\begin{equation}\label{eq3-omega-gap}
0<m-Ca^{p-2}\leq\omega_a
\leq\frac{2c}{a^2}<m.
\end{equation}
Subtracting \eqref{eq3.35} from \eqref{eq3.34}
and testing against $(\Phi(w_n)-z_a)^+$, we obtain
\begin{equation}\label{eq3.36}
\begin{split}
\|(\Phi(w_n)-z_a)^+\|^2
&=\int_{\mathbb R^N}
\bigl(g(x,|\Phi(w_n)|)\Phi(w_n)-g(x,|z_a|)z_a\bigr)
\cdot(\Phi(w_n)-z_a)^+\,dx\\
&\quad+
\bigl(\kappa(\Phi(w_n))\Phi(w_n)-\omega_a z_a,
      (\Phi(w_n)-z_a)^+\bigr)_2
+o_n(1).
\end{split}
\end{equation}

By $(H_2)$,
\begin{equation}\label{eq3-weighted-growth}
g(x,t)t\leq g(x,1)\bigl(t^{p-1}+t^{q-1}\bigr),
\qquad t\geq0,
\quad\text{for a.e. }x\in\mathbb R^N.
\end{equation}
By \eqref{eqg_1}, \eqref{eq3-weighted-growth},
$(H_3)$, and \eqref{eq3-local-convergence},
for every fixed $R>0$,
\begin{equation}\label{eq3-local-part}
\begin{aligned}
&\left|
\int_{|x|\leq R}
\bigl(g(x,|\Phi(w_n)|)\Phi(w_n)-g(x,|z_a|)z_a\bigr)
\cdot(\Phi(w_n)-z_a)^\pm\,dx
\right|\\
&\quad\leq
C\left(
\|\Phi(w_n)-z_a\|_{L^p(B_R)}
+\|\Phi(w_n)-z_a\|_{L^q(B_R)}
\right)
\longrightarrow0
\quad\text{as }n\to\infty.
\end{aligned}
\end{equation}

By \eqref{eq3-weighted-growth}, \eqref{eq_bound-1},
\eqref{eq_bound-2}, \eqref{eq3-weak-convergence},
and H\"older's inequality,
\begin{equation}\label{eq3-tail-part}
\begin{aligned}
&\left|
\int_{|x|>R}
\bigl(g(x,|\Phi(w_n)|)\Phi(w_n)-g(x,|z_a|)z_a\bigr)
\cdot(\Phi(w_n)-z_a)^\pm\,dx
\right|\leq
C\operatorname*{ess\,sup}_{|x|>R}g(x,1),
\end{aligned}
\end{equation}
where $C>0$ is independent of $n$ and $R$.
By \eqref{eq3-local-part}, \eqref{eq3-tail-part},
and $(H_3)$, letting $n\to\infty$ and then $R\to\infty$,
we obtain
\begin{equation}\label{eq3.37}
\int_{\mathbb R^N}
\bigl(g(x,|\Phi(w_n)|)\Phi(w_n)-g(x,|z_a|)z_a\bigr)
\cdot(\Phi(w_n)-z_a)^\pm\,dx
\to0.
\end{equation}

By \eqref{eq3-weak-convergence}, \eqref{eq3-mass-bound},
and \eqref{eq3-omega-gap},
\begin{equation}\label{eq3.38}
\begin{split}
&\left|
\bigl(\kappa(\Phi(w_n))\Phi(w_n)-\omega_a z_a,
      (\Phi(w_n)-z_a)^+\bigr)_2
\right|\\
&\quad=
\left|
\bigl(\kappa(\Phi(w_n))-\omega_a\bigr)
\bigl(\Phi(w_n),(\Phi(w_n)-z_a)^+\bigr)_2
+\omega_a\|(\Phi(w_n)-z_a)^+\|_2^2
\right|\\
&\quad\leq
o_n(1)\|(\Phi(w_n)-z_a)^+\|_2
+\omega_a\|(\Phi(w_n)-z_a)^+\|_2^2.
\end{split}
\end{equation}
By \eqref{eq2.5} and \eqref{eq3.36}--\eqref{eq3.38},
\[
\left(1-\frac{\omega_a}{m}\right)
\|(\Phi(w_n)-z_a)^+\|^2
\leq o_n(1).
\]
Hence, by \eqref{eq3-omega-gap},
\begin{equation}\label{eq3-positive-convergence}
\Phi^+(w_n)\to z_a^+
\quad\text{in }E.
\end{equation}
Subtracting \eqref{eq3.35} from \eqref{eq3.34},
testing against $-(\Phi(w_n)-z_a)^-$,
and using \eqref{eq3.37}, \eqref{eq3-weak-convergence},
and \eqref{eq3-mass-bound}, we obtain
\[
\begin{aligned}
&\|(\Phi(w_n)-z_a)^-\|^2
+\omega_a\|(\Phi(w_n)-z_a)^-\|_2^2\\
&\quad=
-\int_{\mathbb R^N}
\bigl(g(x,|\Phi(w_n)|)\Phi(w_n)-g(x,|z_a|)z_a\bigr)
\cdot(\Phi(w_n)-z_a)^-\,dx\\
&\qquad-
\bigl(\kappa(\Phi(w_n))-\omega_a\bigr)
\bigl(\Phi(w_n),(\Phi(w_n)-z_a)^-\bigr)_2
+o_n(1)\\
&\quad=o_n(1).
\end{aligned}
\]
Hence, by \eqref{eq3-omega-gap},
\begin{equation}\label{eq3-negative-convergence}
\Phi^-(w_n)\to z_a^-
\quad\text{in }E.
\end{equation}
By \eqref{eq3-positive-convergence}
and \eqref{eq3-negative-convergence},
\begin{equation}\label{eq3-strong-convergence}
\|\Phi(w_n)-z_a\|^2
=\|(\Phi(w_n)-z_a)^+\|^2
+\|(\Phi(w_n)-z_a)^-\|^2
\longrightarrow0.
\end{equation}
By \eqref{eq2.5}, \eqref{eq3-mass-bound},
and \eqref{eq3-strong-convergence},
\begin{equation}\label{eq3-limit-mass}
\begin{aligned}
\|z_a\|_2
&=\lim_{n\to\infty}\|\Phi(w_n)\|_2=a,\\
\|z_a^+\|_2
&=\lim_{n\to\infty}\|\Phi^+(w_n)\|_2
\geq\frac{\sqrt3}{2}a>0.
\end{aligned}
\end{equation}
By the definition of $\Phi$,
\eqref{eq3-strong-convergence}, and \eqref{eq3-limit-mass},
\[
w_n
=a\,\frac{\Phi^+(w_n)}{\|\Phi^+(w_n)\|_2}
\longrightarrow
a\,\frac{z_a^+}{\|z_a^+\|_2}
=:w_a
\quad\text{in }E.
\]
Since $X_a$ is closed in $E$, we have $w_a\in X_a$.
The choice of $a_0$ is independent of $c$.
This completes the proof.
\end{proof}

\medskip

\noindent {\bf The proof of Theorem 1.1. (Existence)} 

As a first step, we assert the existence of $(PS)_{c_{a}}$-sequence of $I$ restricts to
$X_a$.
Let $\{w_n\}\subset W\cap S_a$ be a minimizing sequence
with respect to the $I$ such that
\begin{equation*}
    I(w_n)\to c_a,\qquad I(w_n)\leq c_a+n^{-2}.
\end{equation*}
For $a>0$ small enough, up to a subsequence, we can assume that
\begin{equation*}
    \|w_n\|^2\leq\left(m+\frac{1}{2} a^{\frac{p-2}{2}}\right) a^2,~\forall n \in \mathbb{N}.
\end{equation*}
Using the Ekeland variational principle\cite{Ekeland1974JMAA} on the complete metric space
$X_a$, we obtain $w_n^*\in X_a$ such that $I(w_n^*)\leq I(w_n)$, $I(w_n^*)\leq I(w)+\frac{1}{n}\|w-w_n^*\|$ for all $w\in X_a$, and $\|w_n-w_n^*\|\leq\frac{1}{n}$. In particular, $I(w_n^*)\to c_a$.
Consequently, we can assume
that $w_n^*\rightharpoonup w_a$ in $E^+$ and satisfies
$$
\|w_n^*\|^2\leq\left(m+\frac{3}{4} a^{\frac{p-2}{2}}\right) a^2.
$$
Given $\eta\in T_{w_n^*}=\left\{\eta\in E^+\mid(w_n^*,\eta)_{2}=0\right\}$, there is
$\delta>0$ small enough such that the path $\gamma:(-\delta,
\delta)\longrightarrow S_a \cap E^+$ defined by
$$
\gamma(t)=\sqrt{1-t^2 \frac{\|\eta\|_{2}^2}{a^2}} w_n^*+t \eta.
$$
Since $w_n^*$ lies in the relative interior of $X_a$, we may decrease $\delta$ so that $\gamma(t)\in X_a$ for $|t|<\delta$. In particular, $\gamma\in C^1\left((-\delta,\delta),S_a\cap E^+\right)$ and
satisfies
\begin{equation*}
    \gamma(0)=w_n^*~~~\text{and}~~~\gamma^\prime(0)=\eta.
\end{equation*}
Hence,
\begin{equation*}
    I(\gamma(t))-I(w_n^*)\geq -\frac{1}{n}\left\|\gamma(t)-w_n^*\right\|,
\end{equation*}
and in particular,
\[
\frac{I(\gamma(t))-I(\gamma(0))}{t}
\geq-\frac{1}{n}
\left\|\frac{\gamma(t)-\gamma(0)}{t}\right\|,
\qquad 0<t<\delta.
\]
Since $I_{|X_a}\in C^1$, taking the limit of $t\to 0^+$, we get
\begin{equation*}
    \left\langle I^{\prime}\left(w_n^*\right), \eta\right\rangle\geq -\frac{1}{n}\|\eta\|.
\end{equation*}
Replacing $\eta$ by $-\eta$, we obtain
\begin{equation*}
    |\left\langle I ^{\prime}\left(w_n^*\right), \eta\right\rangle|\leq \frac{1}{n}\|\eta\|,~~\forall \eta
    \in
    T_{w_n^*},
\end{equation*}
 which shows that $\{w_n^*\}$ is a $(PS)_{c_a}$-sequence for $I$ restricted to $X_a$.\medskip

Lemma \ref{Lem3.6} and Lemma \ref{Lem3.8} ensure that there is $z_a$ such
that $\Phi(w_n^*) \rightarrow z_a$ in
$H^{1}\left(\mathbb{R}^N,\R^2\right)$ and
$\kappa(\Phi(w_n^*))\rightarrow \omega_a$ as $n\rightarrow\infty$. Again, by Lemma \ref{Lem3.8} and \eqref{eq3.33},
we can conclude that
\begin{equation*}
    Tz_a=g(x,|z_a|)z_a+\omega_a z_a,
\end{equation*}
and
\begin{equation*}
\|z_a\|_{2}=\lim_{n\to \infty}\|\Phi(w_n^*)\|_{2}=\lim_{n\to
\infty}\|w_n^*\|_{2}=a,
\end{equation*}
the existence is completed.

\medskip

{\bf (Bifurcation)} Note that, from the proof of Lemma \ref{Lem3.8},
we have
\begin{equation*}
\omega_a\leq \frac{2c_a}{a^2}<m.
\end{equation*}
Moreover, by the definition of $\kappa$, for $a$ small enough
\begin{equation*}
\kappa(\Phi(w_n^*))=\frac{\|\Phi^+(w_n^*)\|^2}{\|\Phi^+(w_n^*)\|_{2}^2}-\frac{
\displaystyle\int_{\mathbb{R}^N}g(x,|\Phi(w_n^*)|)\Phi(w_n^*)\cdot \Phi^+(w_n^*)dx}{\|\Phi^+(w_n^*)\|_{2}^2}\geq
m-Ca^{p-2}.
\end{equation*}
Consequently, the above estimate gives that 
\begin{equation*}
    m-Ca^{p-2}\leq \omega_a\leq \frac{2c_a}{a^2}<m,
\end{equation*}
which implies that
\begin{equation*}
\omega_a\to m~~\text{as}~~a\to 0^+.
\end{equation*}
On the other hand,
\begin{equation*}
    \|z_a\|= \lim_{n\to \infty}\|\Phi(w_n^*)\|\leq \lim_{n\to \infty} \sqrt{\|w_n^*\|^2+\|\phi(w_n^*)\|^2}\leq \frac{\sqrt{5m+4}}{2}a\to 0~~\text{as}~~a\to 0^+.
\end{equation*}
By definition of \ref{Def1.1}, $m$ is a bifurcation point on the left of equation
\eqref{eq1.1}.

\section{Multiplicity of Normalized solutions}
The purpose of this section is to prove Theorem~\ref{Thm1.2}
by applying the Ljusternik--\allowbreak Schnirelmann multiplicity theorem
in \cite{jeanjean1992} to the reduced functional $I|_{X_a}$.
We retain the variational framework and notation introduced
in Sections~2 and~3 for system~\eqref{eq1.1}, under
assumptions $(V_1)$ and $(H_0)$--$(H_4)$.

Throughout this section, we fix a sufficiently small mass
$a>0$ and write $X=X_a$.
Let $\Sigma(X)$ denote the family of closed symmetric
subsets of $X$.
For each nonempty $A\in\Sigma(X)$, let $\gamma(A)$ denote
its Krasnoselskii genus of $A$.
We refer to \cite{Willem1996book1} for the definition and
basic properties of the genus, and recall below, without
proof, the properties needed in the subsequent arguments.

\begin{Lem}
Let $A,B\in\Sigma(X)$. Then the following statements hold:
\begin{enumerate}
    \item[(i)]
    If $\gamma(A)\geq 2$, then $A$ contains infinitely many
    distinct points;

    \item[(ii)]
    If $A\subset B$, then $\gamma(A)\leq\gamma(B)$;

    \item[(iii)]
    Let $\Omega\subset\mathbb{R}^k$ be an open bounded symmetric
    neighborhood of $0$; then $\gamma(\partial\Omega)=k$;

    \item[(iv)]
    If there exists an odd continuous mapping
    $\psi:\mathbb{S}^{k-1}\to A$, then $\gamma(A)\geq k$.
\end{enumerate}
\end{Lem}
For each $k\in\mathbb{N}$, we define
\[
\Gamma_k
:=
\left\{
A\in\Sigma(X)\mid \gamma(A)\geq k
\right\}.
\]
With this notation, we recall the following Ljusternik-Schnirelmann type multiplicity theorem, which will be applied to the reduced functional \(I_{|_{X_a}}\) to prove Theorem 1.2.

\medskip
\noindent\textbf{Theorem A} (\ \cite{jeanjean1992,Willem1996book1}).
\begingroup
\itshape
Let $J$ be an even $C^1$ functional on a symmetric open
neighborhood of $X=X_a$ in $E^+\cap S_a$. Assume that $J$
is bounded below on $X$,
\[
\inf_{\partial_{E^+\cap S_a}X}J>\frac12ma^2,
\]
and that $J$ restricted to $X$ satisfies $(PS)_c$ for every
$c<\frac12ma^2$, where derivatives and critical points are
taken along $E^+\cap S_a$. Set
\[
b_k
=
\inf_{A\in\Gamma_k}\sup_{v\in A}J(v),
\qquad k\geq 1.
\]
Then the following hold:
\begin{enumerate}
    \item[(1)]
    $b_k$ is a critical value of $J_{|_X}$ provided that
$
    b_k<\frac{ma^2}{2}.
$

    \item[(2)]
    Denoting by $K_b$ the set of critical points of $J_{|_X}$
    at a level $b\in\mathbb{R}$, if
    \[
    b_k=b_{k+1}=\cdots=b_{k+r-1}
    =:b<\frac{ma^2}{2}
    \]
    for some $k,r\geq 1$, then $\gamma(K_b)\geq r$.
    In particular, $J_{|_X}$ has infinitely many critical points
    at the level $b$ if $r\geq 2$.

    \item[(3)]
    If $b_k<\displaystyle\frac{ma^2}{2}$ for all $k\geq 1$, then there
    exist infinitely many critical points $\{v_a^n\}$ of $J_{|_X}$,
    and these critical points satisfy
  $J(v_a^n)<\displaystyle\frac{ma^2}{2}.$
\end{enumerate}
\endgroup

The standard odd deformation argument applies below $ma^2/2$,
since such sublevels do not meet the relative boundary of $X$.
If $b_k$ were not critical, a small downward deformation would
preserve genus and contradict its minimax definition. If $r$
consecutive levels coincide, removing a symmetric neighborhood
of $K_b$ and using genus subadditivity gives $\gamma(K_b)\geq r$.

For the present functional, $J_0$ is even and $M(-w)=-M(w)$.
The uniqueness in Lemma~\ref{Lem3.5} therefore gives
$\Phi(-w)=-\Phi(w)$, so $I=J_0\circ\Phi$ is even.
Moreover, $I$ is $C^1$ on $W_*$, and the estimate
$I(w)\geq\frac12\|w\|^2-Ca^p$ on $W\cap S_a$
shows that $I$ is bounded below on $X_a$. On its relative boundary,
\[
I(w)\geq\frac{ma^2}{2}+\frac12a^{(p+2)/2}-Ca^p
>\frac{ma^2}{2}
\]
for sufficiently small $a$, independently of the genus index.

\begin{Lem}\label{section4-coompactness}
Assume that $(V_1)$ and $(H_0)$--$(H_4)$ hold.
Then there exists $a_0>0$ such that, for every
$a\in(0,a_0)$, the functional $I_{|_{X_a}}$
satisfies the Palais--Smale condition at every level
$c<\displaystyle\frac{ma^2}{2}$.
\end{Lem}
\begin{proof}
Let $\{w_n\}\subset X_a$ satisfy
\begin{equation}\label{eq-ps-sequence}
I(w_n)\to c<\frac{ma^2}{2},
\qquad
I_{|X_a}^{\prime}(w_n)\to0.
\end{equation}
Set $z_n:=\Phi(w_n)$ and $\kappa_n:=\kappa(z_n)$.
By decreasing $a_0$ if necessary, we may assume that
$a_0<1$ and that Lemma~\ref{Lem3.5} applies.

By the definitions of $X_a$ and $M(w_n)$,
together with \eqref{eq2.5}, we have
\begin{equation}\label{eq-ps-bounds}
\|z_n\|\leq Ca,
\qquad
\|z_n\|_2=a,
\qquad
\|z_n^+\|_2^2\geq\frac34a^2,
\end{equation}
where $C>0$ is independent of $a$ and $n$.
Arguing as in the proof of Lemma~\ref{Lem3.8},
using \eqref{eq3.11}, \eqref{eq3.29},
\eqref{eq3.32} and \eqref{eq-ps-sequence},
we obtain
\begin{equation}\label{eq-ps-equation}
Tz_n-g(x,|z_n|)z_n-\kappa_nz_n\to0
\quad\text{in }E^*.
\end{equation}

By \eqref{eqg_2}, \eqref{eqg_4}, the Sobolev
embeddings, and \eqref{eq-ps-bounds},
\begin{equation}\label{eq-ps-nonlinear-bound}
\left|
\int_{\mathbb R^N}
g(x,|z_n|)z_n\cdot z_n^+\,dx
\right|
\leq C(a^p+a^q)\leq Ca^p.
\end{equation}
Since $z_n^+=\frac{\|z_n^+\|_2}{a}w_n$, it follows
from \eqref{eq2.5} and the definition of $X_a$ that
\[
m\leq
\frac{\|z_n^+\|^2}{\|z_n^+\|_2^2}
=\frac{\|w_n\|^2}{a^2}
\leq m+a^{(p-2)/2}.
\]
Thus, the definition of $\kappa$,
\eqref{eq-ps-bounds}, and
\eqref{eq-ps-nonlinear-bound} yield
\begin{equation}\label{eq-ps-kappa-bounds}
m-Ca^{p-2}
\leq\kappa_n
\leq m+1+Ca^{p-2}.
\end{equation}
Passing to a subsequence, we may assume that
\begin{equation}\label{eq-ps-weak-limit}
z_n\rightharpoonup z\quad\text{in }E,
\qquad
\kappa_n\to\omega\quad\text{in }\mathbb R.
\end{equation}
Testing \eqref{eq-ps-equation} against $z_n$
and using \eqref{eqg_2}, we obtain
\begin{equation}\label{eq-ps-energy-multiplier}
\begin{aligned}
\kappa_n a^2
&=2I(w_n)+2\Psi(z_n)
-\int_{\mathbb R^N}g(x,|z_n|)|z_n|^2\,dx+o(1)\\
&\leq2I(w_n)+o(1).
\end{aligned}
\end{equation}
Combining \eqref{eq-ps-sequence},
\eqref{eq-ps-kappa-bounds},
\eqref{eq-ps-weak-limit}, and
\eqref{eq-ps-energy-multiplier},
and choosing $a_0$ smaller if necessary, we obtain
\begin{equation}\label{eq-ps-spectral-gap}
0<m-Ca^{p-2}
\leq\omega
\leq\frac{2c}{a^2}<m.
\end{equation}
Using $(H_1)$--$(H_3)$ and the 
Sobolev embeddings, as the argument for
\eqref{eq3.37} in the proof of Lemma~\ref{Lem3.8},
we obtain
\begin{equation}\label{eq-ps-nonlinear-convergence}
g(x,|z_n|)z_n\to g(x,|z|)z
\quad\text{in }E^*.
\end{equation}
Consequently, \eqref{eq-ps-equation},
\eqref{eq-ps-weak-limit}, and
\eqref{eq-ps-nonlinear-convergence} imply
\[
Tz=g(x,|z|)z+\omega z.
\]
Setting $d_n:=z_n-z$, we have
\begin{equation}\label{eq-ps-difference}
(T-\omega)d_n\to0
\quad\text{in }E^*.
\end{equation}
Testing \eqref{eq-ps-difference} against $d_n^+$
and $-d_n^-$, respectively, and using
\eqref{eq2.5} and \eqref{eq-ps-spectral-gap}, we get
\begin{equation}\label{eq-ps-strong-estimates}
\begin{aligned}
\left(1-\frac{\omega}{m}\right)\|d_n^+\|^2
&\leq\|d_n^+\|^2-\omega\|d_n^+\|_2^2=o(1),\\
\|d_n^-\|^2
&\leq\|d_n^-\|^2+\omega\|d_n^-\|_2^2=o(1).
\end{aligned}
\end{equation}
Hence $z_n\to z$ in $E$.

Finally, by \eqref{eq-ps-bounds},
$\|z^+\|_2\geq\sqrt3\,a/2>0$, and thus
\[
w_n
=a\,\frac{z_n^+}{\|z_n^+\|_2}
\longrightarrow
a\,\frac{z^+}{\|z^+\|_2}
=:w
\quad\text{in }E.
\]
Since $X_a$ is closed in $E$, we have $w\in X_a$.
This proves the $(PS)_c$ condition.
The choice of $a_0$ is independent of $c$.
\end{proof}

\begin{Lem}\label{Lem4.2}
\rm{(\cite[Lemma 3.6]{Heinz1993NA})}
Let $Z$ be a finite dimensional normed space,
$Y$ an arbitrary normed space, and let $\{P_n\}$
be a sequence of linear maps from $Z$ to $Y$ such that
$$
\lim_{n\to\infty}\|P_n z\|_Y=\|z\|_Z,
\qquad \forall z\in Z.
$$
Then
\begin{itemize}
    \item[$(i)$]
    the convergence $\|P_n z\|_Y\to\|z\|_Z$
    is uniform on bounded subsets of $Z$;
    \item[$(ii)$]
    there exists $n_0\in\mathbb N$ such that
    $P_n$ is injective for every $n\geq n_0$.
\end{itemize}
\end{Lem}
Next, we will construct a map $P_n: Z \rightarrow
H^{1}(\mathbb{R}^{N},\R^2)$ that satisfies Lemma \ref{Lem4.2}
and give some estimates for it. Given $k \in \mathbb{N}$,
we choose a $k$-dimensional subspace $Z$, which is spanned
of the first $k$ Hermite functions in $\mathbb{R}^N$,
and we endow $Z$ with the $L^2$-norm.
We now define the linear maps $P_n: Z \rightarrow
H^{1}(\mathbb{R}^{N},\R^2)$ by
\begin{equation}\label{eq4.1}
(P_n\zeta)(x):=
\frac{1}{n^{N/2}\sqrt{M(|\nu|^2)}}
\nu(x)\zeta\!\left(\frac{x}{n}\right).
\end{equation}
for $\zeta \in Z, n \in \mathbb{N}, x \in \mathbb{R}^N$,
where $\nu$ is defined in Section~3.1.
The same identity $\Delta\nu=V\nu-m\mathcal J\nu$, together
with the rapid decay of Hermite functions, gives
$P_n\zeta\in H^2(\mathbb R^N,\mathbb R^2)$.
In addition, we denote
$Z^1:=\{\zeta \in Z \mid\|\zeta\|_{2}=1\}$.
\begin{Lem}\label{Lem4.3}
     There exist constants $C_0>0$ and $n_0\in \mathbb{N}$ such that
    \begin{equation*}
        \Psi\left(P_n \zeta\right)
        \geq C_0n^{-\frac{N(\alpha-2)}{2}-\tau},
        ~~\forall n>n_0~~\text{and}~~\zeta\in Z^1,
    \end{equation*}
   where $\{P_n\}$ be defined by \eqref{eq4.1}.
\end{Lem}
\begin{proof}
   Since $\nu=(\phi,\psi)$, we have
$
M(|\nu|^2)=M(|\phi|^2+|\psi|^2)>0$
and hence
\[
|P_n\zeta(x)|
=
\frac{|\nu(x)|\,|\zeta(x/n)|}
{n^{N/2}\sqrt{M(|\nu|^2)}}.
\]
Since $Z$ is a finite-dimensional space spanned
by Hermite functions,
\[
\sup_{\zeta\in Z^1}\|\zeta\|_\infty<\infty.
\]
Therefore,
\[
\sup_{\zeta\in Z^1}\|P_n\zeta\|_\infty
\leq
\frac{\|\nu\|_\infty}
{n^{N/2}\sqrt{M(|\nu|^2)}}
\sup_{\zeta\in Z^1}\|\zeta\|_\infty
\longrightarrow0.
\]
Thus, for all sufficiently large $n$,
$\|P_n\zeta\|_\infty\leq t_0$ for every
$\zeta\in Z^1$, and assumption $(H_4)$ gives
$$
\begin{aligned}
\Psi(P_n\zeta)
&\geq L\int_S |x|^{-\tau}|P_n\zeta|^\alpha dx\\
&=\frac{Ln^{-\frac{N\alpha}{2}}}
{M^{\frac{\alpha}{2}}(|\nu|^2)}
\int_S |x|^{-\tau}|\nu(x)|^\alpha
\bigg|\zeta\bigg(\frac{x}{n}\bigg)\bigg|^\alpha dx\\
&=\frac{Ln^{-\frac{N(\alpha-2)}{2}-\tau}}
{M^{\frac{\alpha}{2}}(|\nu|^2)}
\int_{S_n}|y|^{-\tau}|\nu(ny)|^\alpha
|\zeta(y)|^\alpha dy,
\end{aligned}
$$
where
\[
S_n:=\{y\in\mathbb R^N:ny\in S\}.
\]
For $n\geq1$, we have $S\subset S_n$.
Using Proposition \ref{Proposition 3.1}, we know that
\begin{align}\label{eq4.2}
\frac{M^{\frac{\alpha}{2}}(|\nu|^2)}
{Ln^{-\frac{N(\alpha-2)}{2}-\tau}}
\Psi(P_n\zeta)
&\geq
\int_{\mathbb R^N}|\nu(ny)|^\alpha
\chi_S(y)|y|^{-\tau}|\zeta(y)|^\alpha dy\\
&\to
M(|\nu|^\alpha)
\int_{\mathbb R^N}g(y)\,dy,\nonumber
\end{align}
where
$$
g(y):=\chi_S(y)|y|^{-\tau}|\zeta(y)|^\alpha,
$$
and $\chi_S$ is the characteristic function of the set $S$.

\medskip

Note that $\zeta\in Z$ which is slightly different from
the proof of Lemma \ref{Lem3.2}, we now consider the
linear maps
$\Lambda_n:Z\rightarrow L^\alpha\left(S;\ |y|^{-\tau}dy\right)$
given by
    $$
    (\Lambda_n\zeta)(x):=|\nu(nx)|\zeta(x),
\quad x\in S,
    $$
where the norm on $Z$ is now given by
    $$
\|\zeta\|_*:=
\left(
M(|\nu|^\alpha)
\int_S|y|^{-\tau}|\zeta(y)|^\alpha dy
\right)^{\frac1\alpha}.
    $$
Due to the fact that every $\zeta\in Z$ is real analytic,
we can assume that for any $\zeta\in Z^1$,
    \begin{equation*}
\zeta(x)\neq0~\text{a.e.}~\text{on}~S.
    \end{equation*}
As a consequence, $\|\cdot\|_*$ is indeed a norm
and the subspace $Z$ dimension is finite which leads
to the equivalence of the norms on $Z$,
and we know that there exists a positive constant $C_5$
such that
    \begin{equation}\label{eq4.3}
\|\zeta\|_*\geq C_5,~~\forall\zeta\in Z^1.
    \end{equation}
By \eqref{eq4.2},
\[
\|\Lambda_n\zeta\|_{L^\alpha(S;|y|^{-\tau}dy)}
\longrightarrow\|\zeta\|_*
\qquad\text{for every }\zeta\in Z.
\]
By Lemma~\ref{Lem4.2}, this convergence is uniform on $Z^1$.
Together with \eqref{eq4.3}, this implies that, for all
sufficiently large $n$,
\[
\|\Lambda_n\zeta\|_{L^\alpha(S;|y|^{-\tau}dy)}
\geq\frac{C_5}{2}
\qquad\text{for all }\zeta\in Z^1.
\]
Consequently,
\[
\Psi(P_n\zeta)
\geq
\frac{LC_5^\alpha}
{2^\alpha M^{\alpha/2}(|\nu|^2)}
n^{-\frac{N(\alpha-2)}2-\tau}
=:C_0n^{-\frac{N(\alpha-2)}2-\tau}.
\]
for all $n\geq n_0$ and $\zeta\in Z^1$,
after increasing $n_0$ if necessary.
\end{proof}

\begin{Lem}\label{Lem4.4}
For any $k\in \mathbb{N}$ and let $\{P_n\}$
be defined by \eqref{eq4.1}, then
\begin{itemize}
    \item[$(i)$]
    $\lim\limits_{n\rightarrow\infty}
    \left\|P_n\zeta\right\|_{2}=\|\zeta\|_{2}$
    for all $\zeta\in Z$;

    \item[$(ii)$]
    there exists constant $C$ such that,
    for $n$ large enough
    \begin{equation*}
    0<\Psi(P_n\zeta)\leq C~
    \text{ for all }\zeta\in Z^1;
    \end{equation*}

    \item[$(iii)$]
    \begin{equation*}
    \frac{\sup\limits_{\zeta\in Z^1}
    \left|((T-mI)P_n\zeta,P_n\zeta)_{2}\right|}
    {\inf\limits_{\zeta\in Z^1}\Psi(P_n\zeta)}
    \rightarrow0
    \text{ as }n\rightarrow\infty,
    \end{equation*}
    and
    \begin{equation*}
    \frac{\sup\limits_{\zeta\in Z^1}
    \|(T-mI)P_n\zeta\|_{2}^{2}}
    {\inf\limits_{\zeta\in Z^1}\Psi(P_n\zeta)}
    \rightarrow0
    \text{ as }n\rightarrow\infty.
    \end{equation*}
\end{itemize}
\end{Lem}

\begin{proof}
$(i)$ Clearly, it follows from
Proposition \ref{Proposition 3.1} that
\begin{equation*}
\int_{\mathbb R^N}|P_n\zeta|^2dx
=\frac{1}{M(|\nu|^2)}
\int_{\mathbb R^N}\zeta^2(x)|\nu(nx)|^2dx
\longrightarrow
\int_{\mathbb R^N}|\zeta|^2dx
\quad\text{as }n\to\infty.
\end{equation*}

$(ii)$ By Lemma \ref{Lem4.3} we know that
$\Psi(P_n\zeta)>0$ for all $\zeta\in Z^1$,
so we only need to prove that $P_n(Z^1)$ is
uniformly bounded in $H^1(\mathbb R^N,\mathbb R^2)$.
For all $\zeta\in Z^1$, since all norms on $Z$
are equivalent, there exists constant $C_7>0$
such that
\begin{equation}\label{eq4.4}
\|\nabla\zeta\|_2^2+\|\Delta\zeta\|_2^2\leq C_7.
\end{equation}
Since $\nu$ and its first-order derivatives are bounded,
for every $\zeta\in Z^1$ and $n\geq1$, we have
\begin{equation*}
\|P_n\zeta\|_2^2
\leq\frac{\|\nu\|_\infty^2}{M(|\nu|^2)}
\end{equation*}
and
\begin{equation*}
\begin{split}
\|\nabla(P_n\zeta)\|_2^2
&\leq
\frac{2}{M(|\nu|^2)}
\left(
\|\nabla\nu\|_\infty^2\|\zeta\|_2^2
+n^{-2}\|\nu\|_\infty^2\|\nabla\zeta\|_2^2
\right)\leq C,
\end{split}
\end{equation*}
where the last inequality follows from \eqref{eq4.4}.

Moreover, by $(i)$ and Lemma~\ref{Lem4.2},
the convergence $\|P_n\zeta\|_2\to1$ is uniform
on $Z^1$. Hence, for all sufficiently large $n$,
\begin{equation*}
\|P_n\zeta\|_2\geq\frac12,
\qquad \forall\zeta\in Z^1.
\end{equation*}
Combining these estimates with \eqref{eq2.5},
we obtain constants $C_1',C_2'>0$ such that,
for $n$ large enough,
\begin{equation*}
0<C_1'\leq\|P_n\zeta\|\leq C_2',
\qquad\forall\zeta\in Z^1.
\end{equation*}
Therefore, \eqref{eqg_4} and the Sobolev embeddings yield
\begin{equation*}
0<\Psi(P_n\zeta)
\leq C\left(
\|P_n\zeta\|_p^p+\|P_n\zeta\|_q^q
\right)
\leq C,
\qquad\forall\zeta\in Z^1.
\end{equation*}

$(iii)$ For any $\zeta\in Z^1$, in view of
\eqref{eq4.4} and repeating the arguments of
Lemma \ref{Lem3.3}, there exists a positive
constant $C$ independent of $\zeta$ such that
\begin{equation}\label{eq4.5}
|((T-mI)P_n\zeta,P_n\zeta)_{2}|
\leq\frac{C}{n^2}
\end{equation}
and
\begin{equation}\label{eq4.6}
\|(T-mI)P_n\zeta\|_{2}^{2}
\leq\frac{C}{n^2}.
\end{equation}
By Lemma \ref{Lem4.3}, $\eqref{eq4.5}$
and $\eqref{eq4.6}$, we have
\begin{equation*}
\frac{\sup\limits_{\zeta\in Z^1}
\|(T-mI)P_n\zeta\|_{2}^{2}}
{\inf\limits_{\zeta\in Z^1}\Psi(P_n\zeta)}
\leq Cn^{\frac{N(\alpha-2)}{2}+\tau-2}
\to0
\end{equation*}
and
\begin{equation*}
\frac{\sup\limits_{\zeta\in Z^1}
\left|((T-mI)P_n\zeta,P_n\zeta)_{2}\right|}
{\inf\limits_{\zeta\in Z^1}\Psi(P_n\zeta)}
\leq Cn^{\frac{N(\alpha-2)}{2}+\tau-2}
\to0
\end{equation*}
as $n\to\infty$ since
$\tau\in(0,\frac{4-N(\alpha-2)}{2})$
and the Lemma is proven.
\end{proof}

\begin{Lem}\label{Lem4.5} 
For any $k\in \mathbb{N}$,
there exists a sequence $\left\{X_n^k\right\}_{n=1}^\infty$ of
$k$-dimensional linear subspaces of $E^+$ such that
$X_n^k\setminus\{0\}\subset W$ and
\begin{equation}\label{eq++}
\lim\limits_{n \rightarrow \infty}\frac{\sup\limits_{w\in S_n^k}
\left(\|w\|^2-m\right)}{\inf\limits_{w\in S_n^k}\Psi(w)} 
=0,
\end{equation}
where
    $$
S_n^k:=\left\{u \in X_n^k \mid\|u\|_{2}=1\right\}.
    $$
\end{Lem}
\begin{proof}
    To prove the existence of $\left\{X_n^k\right\}$, we need to make full use of Lemma \ref{Lem4.2}. In fact,  fix $\zeta \in Z^1$ and let $u_n:=P_n\zeta$. \eqref{eq4.6} gives 
            \begin{equation*}
        \|(T-mI) u_n\|_{2}\to 0~\text{as}~n\to \infty.
    \end{equation*}
and therefore
    \begin{equation}\label{eq5.21+}
    \|(T-mI) u_n\|_{2}\|u_n^-\|_{2}\geq |((T-mI) u_n,u_n^-)_{2}|=\|u_n^-\|^2+m\|u_n^-\|_{2}^2.
\end{equation}
Using \eqref{eq5.21+} we know that 
            \begin{equation}\label{eq4.7}
    \lim\limits_{n\to \infty}\|u_n^-\|_{2}= \lim\limits_{n\to \infty}\|u_n^-\|= 0\ \text{and}\ \|u_n^-\|_{2}\leq \frac{1}{m}\|(T-mI)u_n\|_{2}
\end{equation}
Furthermore, we have $\|u_n\|_{2}=\|P_n
\zeta\|_{2}\to \|\zeta\|_{2}=1$ as $n\to \infty$. Combining these two facts above, we arrive at the convergence
            \begin{equation}\label{eq4.9}
    \|(P_n \zeta)^+\|_2=\|u_n^+\|_{2}\to 1~\text{as}~n\to \infty.
\end{equation}
For any $\zeta\in Z$, \eqref{eq4.9} makes the operator $P_n^+:Z\rightarrow L^2(\mathbb{R}^N,\R^2)$ given by  
\begin{equation*}
P_n^+ \zeta=(P_n \zeta)^+
\end{equation*}
satisfies the assumption of Lemma \ref{Lem4.2}, let us assume without loss of
generality that $P_n^+$ is injective for every $n$. If $X_n^k:=P_n^+(Z)$, then $$\dim X_n^k=k.$$ In addition, by \eqref{eq4.5} we know that
    \begin{equation*}
                |((T-mI)P_n \zeta,P_n \zeta)_{2}|\to 0~\text{as}~n\to \infty,
    \end{equation*}
that is,
    \begin{equation*}
\|u_n^+\|^2-\|u_n^-\|^2-m\|u_n\|^2_{2}\to 0~\text{as}~n\to \infty,
\end{equation*}
which implies that
            \begin{equation*}
\|u_n^+\|\to \sqrt{m}~\text{as}~n\to \infty,
\end{equation*}
These convergences are uniform for $\zeta\in Z^1$: this follows
from Lemma~\ref{Lem4.2} for the $L^2$ norms, and from the
uniform estimates \eqref{eq4.5}--\eqref{eq4.6} for the residuals.
Consequently, $X_n^k\setminus\{0\}\subset W$ for all sufficiently
large $n$ (with $k$ fixed). Moreover,
        \begin{equation*}
S_n^k=\left\{\frac{P_n^+\zeta}{\|P_n^+\zeta\|_{2}}\bigg|\zeta\in
Z^1\right\}.
    \end{equation*}
    
We conclude the proof of this Lemma by proving \eqref{eq++}. For any $w_n\in S_n^k$, there exists $\zeta \in
Z^1$ such that $w_n=\frac{P_n^+\zeta}{\|P_n^+\zeta\|_{2}}$. Let
$u_n=P_n\zeta$ again, similarly to the arguments of Lemma \ref{Lem3.4}, we may assume $\frac{1}{2}\leq \|u_n^+\|_{2}
\leq 2$, then 
    \begin{equation}\label{eq4.10}
\Psi(w_n)\geq
2^{-q}\bigl(2^{1-q}\Psi(u_n)-\Psi(u_n^-)\bigr),
\end{equation}
where the proof here uses \eqref{eqg_3} and
the same method as \eqref{eq3.13}.
By \eqref{eqg_4}, \eqref{eq4.7} and Lemma \ref{Lem4.4}, for
$n$ large, we infer that
    \begin{equation}\label{eq4.11}
\Psi(u_n^-)\leq C\|u_n^-\|^p\leq C \|(T-mI)u_n\|_{2}^p\leq
2^{-q}\Psi(u_n).
\end{equation}
Thus, \eqref{eq4.10} and \eqref{eq4.11} yields
    \begin{equation}\label{eq4.12}
    \Psi(w_n)\geq 2^{-2q}\Psi(u_n).
\end{equation}
Furthermore, by \eqref{eq5.21+} and \eqref{eq4.7}, we have
\begin{equation}\label{eq4.13}
\begin{split}
\|w_n\|^2-m
&=\frac{((T-mI)u_n^+,u_n^+)_{2}}
{\|u_n^+\|_{2}^2}\\
&=\frac{((T-mI)u_n,u_n)_{2}
+\|u_n^-\|^2+m\|u_n^-\|_{2}^2}
{\|u_n^+\|_{2}^2}\\
&\leq 4\left|((T-mI)u_n,u_n)_{2}\right|
+\frac{4}{m}\|(T-mI)u_n\|_{2}^2.
\end{split}
\end{equation}
Since $P_n$ is injective and combined with \eqref{eq4.12} and
\eqref{eq4.13}, it can be inferred that
   \begin{equation*}
\begin{aligned}
\sup_{w\in S_n^k}\left(\|w\|^2-m\right)
&\leq
4\sup_{\zeta\in Z^1}
\left|((T-mI)P_n\zeta,P_n\zeta)_{2}\right|+\frac{4}{m}
\sup_{\zeta\in Z^1}\|(T-mI)P_n\zeta\|_{2}^2.
\end{aligned}
\end{equation*}
and
    \begin{equation*}
\inf_{w\in S_n^k}\Psi(w)
\geq 2^{-2q}\inf_{\zeta\in Z^1}\Psi(P_n\zeta).
\end{equation*}
Hence, Lemma \ref{Lem4.4} implies the conclusion of Lemma
\ref{Lem4.5} holds.
\end{proof}

We next establish an upper bound for \(I(a\xi)\) with a smallness threshold for \(a\) independent of \(\xi\). Combined with Lemma 4.6, this estimate will allow us to show that all minimax levels \(b_k\) lie below \(ma^2/2\) for the same fixed sufficiently small mass \(a\).
\begin{Lem}\label{lem-section4+}
There exists a constant $a_0\in(0,1)$ such that,
for every $a\in(0,a_0)$ and every $\xi\in W$
with $\|\xi\|_2=1$, we have
\begin{equation}\label{eq-uniform-reduced}
I(a\xi)\leq
\frac{a^2}{2}\|\xi\|^2
-2^{1-2q}a^q\Psi(\xi).
\end{equation}
\end{Lem}
\begin{proof}
Let $\xi\in W$ with $\|\xi\|_2=1$.
By the definition of $W$, we have
\begin{equation}\label{eq-uniform-xi-bound}
\|\xi\|^2<m+1.
\end{equation}
In view of \eqref{eq3.6} and
\eqref{eq-uniform-xi-bound}, we may apply
Lemma~\ref{Lem3.5} with $w=a\xi$ and choose the
smallness threshold for $a$ independently of $\xi$.

For simplicity, denote $\phi(a\xi)$ by $\phi$
throughout this proof, and set
\[
s:=\frac{\sqrt{a^2-\|\phi\|_2^2}}{a}.
\]
By the definition of $M(a\xi)$ and \eqref{eq2.5},
\begin{equation}\label{eq-uniform-phi-bound}
\|\phi\|\leq\frac{\sqrt m}{2}a,
\qquad
\|\phi\|_2\leq\frac{a}{2}.
\end{equation}
Consequently, \eqref{eq3.2}, \eqref{eq3.9}, and
\eqref{eq-uniform-phi-bound} yield
\begin{equation}\label{eq-uniform-decomposition}
\frac{\sqrt3}{2}\leq s\leq1,
\qquad
\Phi(a\xi)=sa\xi+\phi.
\end{equation}
Using the convexity of $\Psi$, \eqref{eqg_3}, and
\eqref{eq-uniform-decomposition}, we obtain
\begin{equation}\label{eq-uniform-nonlinear-lower}
\begin{aligned}
\Psi(\Phi(a\xi))
&\geq
2\Psi\left(\frac{sa\xi}{2}\right)-\Psi(-\phi)\\
&\geq
2\left(\frac{sa}{2}\right)^q\Psi(\xi)-\Psi(-\phi)\\
&\geq
2^{1-2q}a^q\Psi(\xi)-\Psi(-\phi),
\end{aligned}
\end{equation}
where we have used $0<a<1$ and $s\geq1/2$.

On the other hand, \eqref{eqg_4}, the Sobolev
embeddings, and \eqref{eq-uniform-phi-bound} imply
\begin{equation}\label{eq-uniform-phi-growth}
\begin{aligned}
\Psi(-\phi)
&\leq C\bigl(\|\phi\|^p+\|\phi\|^q\bigr)\\
&=C\bigl(\|\phi\|^{p-2}
         +\|\phi\|^{q-2}\bigr)\|\phi\|^2\\
&\leq
C\left[
\left(\frac{\sqrt m}{2}a\right)^{p-2}
+
\left(\frac{\sqrt m}{2}a\right)^{q-2}
\right]\|\phi\|^2,
\end{aligned}
\end{equation}
where $C>0$ is independent of $a$ and $\xi$.
Since $p,q>2$, we can decrease $a_0$, independently
of $\xi$, so that
\[
C\left[
\left(\frac{\sqrt m}{2}a\right)^{p-2}
+
\left(\frac{\sqrt m}{2}a\right)^{q-2}
\right]\leq\frac12
\]
for every $a\in(0,a_0)$.
It follows from \eqref{eq-uniform-phi-growth} that
\begin{equation}\label{eq-uniform-absorption}
-\frac12\|\phi\|^2+\Psi(-\phi)\leq0.
\end{equation}

By the definition of $I$ and
\eqref{eq-uniform-decomposition},
\eqref{eq-uniform-nonlinear-lower}, and
\eqref{eq-uniform-absorption}, we conclude that
\[
\begin{aligned}
I(a\xi)
&=\frac{s^2a^2}{2}\|\xi\|^2
-\frac12\|\phi\|^2-\Psi(\Phi(a\xi))\\
&\leq
\frac{a^2}{2}\|\xi\|^2
-2^{1-2q}a^q\Psi(\xi)
-\frac12\|\phi\|^2+\Psi(-\phi)\\
&\leq
\frac{a^2}{2}\|\xi\|^2
-2^{1-2q}a^q\Psi(\xi),
\end{aligned}
\]
which proves \eqref{eq-uniform-reduced}.
The choice of $a_0$ is independent of $\xi$.
In particular, the same threshold applies to
$\xi\in S_n^k$, independently of $k$ and $n$.
\end{proof}

\begin{Lem}\label{Lem4.6}
 There exists $a_0>0$ such that, for every $a\in(0,a_0)$
and every $k\in\mathbb{N}$, we have
    \begin{equation*}
        b_k<\frac{ma^2}{2},
    \end{equation*}
where $b_k:=\inf\limits_{A\in \Gamma_k}\sup\limits_{w\in A}I(w)$, $\Gamma_k:=\left\{A \in \Sigma(X)\mid \gamma(A) \geq k\right\}$, $k\geq 1$.
\end{Lem}

\begin{proof}
    We first show that the smallness condition on \(a\) can be chosen independently of \(k\). Let $a_0$ be as in Lemma~\ref{lem-section4+}. Fix \(a\in(0,a_0)\) and \(k\in\mathbb N\). By Lemma~\ref{Lem4.5}, there exists \(n=n(k,a)\) sufficiently large such that
\begin{equation}
\sup_{\xi\in S_n^k}
\frac{\|\xi\|^2-m}{\Psi(\xi)}
<2^{2-2q}a^{q-2}.
\label{eq-minimax-test-ratio}
\end{equation}
Here \(n\) is allowed to depend on \(k\) and \(a\).
By \eqref{eq4.5}, \eqref{eq4.6}, and \eqref{eq4.13},
we may choose $n$ sufficiently large so that, in addition,
\[
\sup_{\xi\in S_n^k}\bigl(\|\xi\|^2-m\bigr)
\leq a^{(p-2)/2}.
\]
Thus, for every $\xi\in S_n^k$, we have
\[
\|a\xi\|_2=a,
\qquad
\|a\xi\|^2
=a^2\|\xi\|^2
\leq\bigl(m+a^{(p-2)/2}\bigr)a^2.
\]
By the definition of $X_a$, it follows that
$aS_n^k\subset X_a$.
For every \(\xi\in S_n^k\), Lemma \ref{lem-section4+} and \eqref{eq-minimax-test-ratio} imply
\[
\begin{aligned}
I(a\xi)
&\leq \frac{a^2}{2}\|\xi\|^2
-2^{1-2q}a^q\Psi(\xi)\\
&<\frac{ma^2}{2}.
\end{aligned}
\]
The set $S_n^k(a)=aS_n^k$ is compact and symmetric, with
$\gamma(S_n^k(a))=k$. Hence $S_n^k(a)\in\Gamma_k$, and
the continuity of $I$ gives
\[
b_k
\leq \max_{w\in S_n^k(a)} I(w)
=\max_{\xi\in S_n^k} I(a\xi)
<\frac{ma^2}{2}.
\]
The constant $a_0$ is independent of $k$.
\end{proof}

\noindent\textbf{Proof of Theorem~\ref{Thm1.2}.}
Choose $a_0>0$ sufficiently small so that
Lemmas~\ref{section4-coompactness} and~\ref{Lem4.6} hold,
and fix $a\in(0,a_0)$.
Theorem~A then gives infinitely many critical points of $I_{|_{X_a}}$,
which yield normalized solutions of \eqref{eq1.1} through $\Phi$.
By \eqref{eq3.2} and \eqref{eq3.9}, we have
\[
w=a\frac{\Phi(w)^+}{\|\Phi(w)^+\|_2},
\qquad w\in X_a.
\]
Hence $\Phi$ is injective on $X_a$. 
Therefore, distinct critical points of $I_{|X_a}$
yield distinct normalized solutions of \eqref{eq1.1}.
The bifurcation conclusion follows as in the proof of
Theorem~\ref{Thm1.1}.

\hfill$\blacksquare$
\par
{\bf Acknowledgements.}
The authors would like to thank the referees for
their careful reading of the manuscript and their
valuable comments and suggestions, which have
helped improve the paper.

{\bf Conflict Of Interest Statement.} The authors declare that they have no conflicts of interest.
\par
{\bf Data Availability Statement.} Our manuscript has no associated data.

\end{document}